\documentclass[a4paper]{article}

\usepackage[english]{babel}
\usepackage[utf8x]{inputenc}
\usepackage[T1]{fontenc}

\usepackage{titlesec}
\titleformat*{\section}{\large\bfseries}
\usepackage[a4paper,top=2cm,bottom=2.5cm,left=2.3cm,right=2.3cm,marginparwidth=1.75cm]{geometry}

\usepackage[colorinlistoftodos]{todonotes}
\usepackage[colorlinks=true, allcolors=blue]{hyperref}
\usepackage{xy}
\usepackage[affil-it]{authblk}
\usepackage{amssymb}
\usepackage{cancel}
\usepackage{epstopdf}
\usepackage{verbatim}
\usepackage{amsmath,amscd}
\usepackage{graphicx}
\usepackage{subcaption}
\usepackage{amsthm}
\usepackage{tikz}
\usetikzlibrary{cd}
\usepackage{pgf}
\usepackage{mathrsfs}
\usetikzlibrary{arrows}
\usepackage[toc,page]{appendix}
\usepackage{caption}
\usepackage[colorinlistoftodos]{todonotes}

\newtheorem{theorem}{Theorem}[section]

\makeatletter
\newtheorem*{rep@theorem}{\rep@title}
\newcommand{\newreptheorem}[2]{%
\newenvironment{rep#1}[1]{%
 \def\rep@title{#2 \ref{##1}}%
 \begin{rep@theorem}}%
 {\end{rep@theorem}}}
\makeatother

\newtheorem{lemma}[theorem]{Lemma}
\newreptheorem{theorem}{Theorem}

\theoremstyle{definition}

\newtheorem{prop}[theorem]{Proposition}
\newtheorem{mydef}[theorem]{Definition}

\newtheorem{rem}[theorem]{Remark}
\newtheorem{exa}[theorem]{Example}

\newtheorem{lma}[theorem]{Lemma}

\title{Faithful tropicalizations of elliptic curves using minimal models and inflection points}%A tropical decomposition of the tame fundamental group%Tame tropical fundamental groups for algebraic curves}%Semistable coverings and torsion points on elliptic curves}
\author{Paul Alexander Helminck}
\affil{University of Bremen,\\
ALTA, Institute for Algebra, Geometry, Topology and their Applications}
\date{September 29, 2019}

\begin{document}
\maketitle

\begin{figure}[h!]

\centering
\includegraphics[width=9cm]{{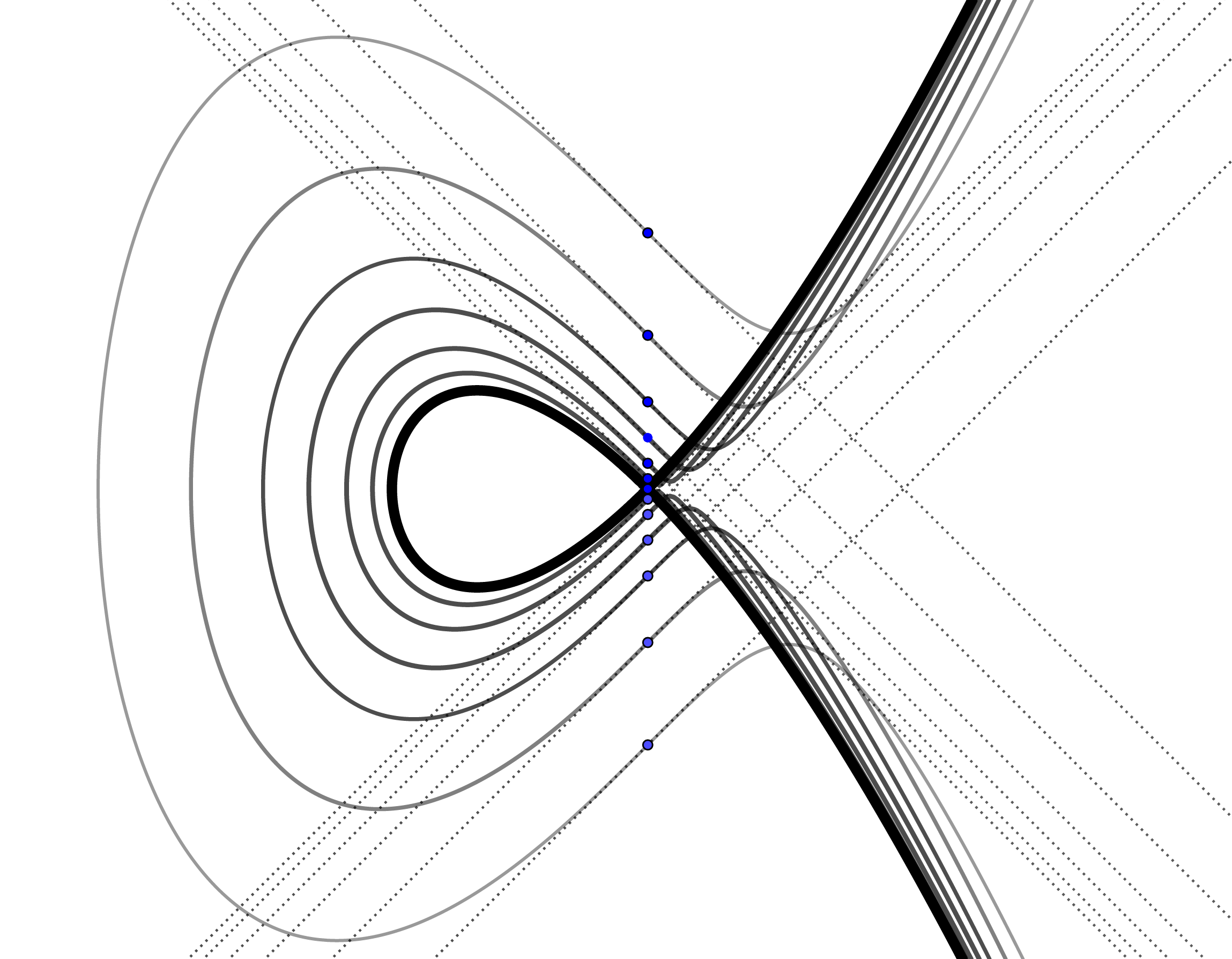}}
\caption{\label{PlaatjeDegTrop3} {\small{A family of elliptic curves degenerating to a nodal curve. Every member in the family has two blue {\it{inflection points}} for which the corresponding dotted tangent lines intersect the curve only at that point. %the rest of the curve only at that point. 
As the curves get closer to the nodal curve, the tangent lines converge to the two distinct tangent directions of the singular curve.}}} 
%at which the tangent line  %The idea of our proof of Theorem \ref{MainTheorem} %The proof of Theorem \ref{MainTheorem} follows  }%
\end{figure}

\begin{abstract}
We give an elementary proof of the fact that any elliptic curve $E$ over an algebraically closed non-archimedean field $K$ with residue characteristic $\neq{2,3}$ and with $v(j(E))<0$ admits a tropicalization that contains a cycle of length $-v(j(E))$. 
We first define an adapted form of minimal models over non-discrete valuation rings and we recover several well-known theorems from the discrete case. Using these, we create an explicit family of marked elliptic curves $(E,P)$, where $E$ has multiplicative reduction and $P$ is an inflection point that reduces to the singular point on the reduction of $E$.  We then follow the strategy as in \cite[Theorem 6.2]{BPR11} and construct an embedding such that its tropicalization contains a cycle of length $-v(j(E))$. We call this a numerically faithful tropicalization. A key difference between this approach and the approach in \cite{BPR11} is that we do not require any of the analytic theory on Berkovich spaces such as the {\it{Poincar\'{e}-Lelong formula}} or \cite[Theorem 5.25]{BPR11} to establish the numerical faithfulness of this tropicalization. 
%and we show that the tropicalization of this embedding contains a cycle of length $-v(j(E))$.    %relate the reduction of a three-torsion%suitable three-torsion points %suitable embedding into $\ma%, which allows us to reduce elliptic curves%which bypasses the analytic theory usually used in this context.
%We do this without relying on any analytic theory, using only an adapted form of minimal models for non-Noetherian valuation rings.
\end{abstract}

\section{Introduction}\label{Introduction}

In this paper, we study the tropicalization of one-parameter families of algebraic curves. The tropicalization process takes such a one-parameter family and assigns a piecewise-linear limit to it. Over the complex numbers, this limit  has been known for some time as the {\it{logarithmic limit set}} or the  {\it{nonarchimedean amoeba}} assigned to the family, see \cite{Bergman1971} and \cite[Definition 1.1.1]{Einsiedler2006}.  This process has proven to be interesting, because the piecewise limit can retain much of the geometry of each of its members if the family is chosen carefully enough. In this paper, the families we will be aiming for are those that give rise to {\it{faithful tropicalizations}}. In the special case of a family of elliptic curves, we will show that we can modify the family such that the new family defines a faithful tropicalization.

Algebraically, the notion of a one-parameter family of complex curves can be viewed as a curve defined over the field of Puiseux series $K$. These Puiseux series are power series over $\mathbb{C}$ in one variable $t$ of the form 
\begin{equation}
f(t)=c_{1}t^{a_{1}}+c_{2}t^{a_{2}}+...
\end{equation} 
where the $c_{i}$ are nonzero complex numbers for all $i$, and $a_{1}<a_{2}<...$ are rational numbers that have a common denominator. %in which we allow fractional (and negative) exponents whose denominators are bounded. % and with only finitely many negative exponents. 
For instance, the formal power series 
\begin{equation}
\dfrac{1}{(1-t^{1/2})}=1+t^{1/2}+t+t^{3/2}+t^{2}+...
\end{equation} 
is contained in $K$. An example of an algebraic curve over this field is then given by the equation 
\begin{equation}
y^2=x^3+(x-t)^2. 
\end{equation} We view this curve over $K$ as giving a family of curves by evaluating $t$ at complex numbers $t_{0}\in\mathbb{C}$. The corresponding picture over $\mathbb{R}$ can be found in Figure \ref{PlaatjeDegTrop3} for some small values of $t$. % this particular family and some values of $t$. 

We are now interested in the degeneration of this family to $t=0$, where the curve becomes a singular curve. To study this process, we consider the valuation map $v:{K}^{*}\rightarrow{\mathbb{Q}}$, which sends a power series $f$ to the lowest power of $t$ in its power series expansion. For instance, we have $v(t^{-1/2})=-1/2$ and $v(t^2+t^3)=2$.  If we now apply this map to points on an algebraic curve with values in ${K}^{*}$, then we obtain our first notion of a {\it{tropicalization}}, see Section \ref{Tropicalizations} for the general construction. The reader can find an example of a tropicalization in Figure \ref{Plaatje1}. In this case, we have tropicalized a family of {\it{elliptic curves}}, which are complex curves whose topological genus is one. As the reader can see from Figure \ref{Plaatje1}, this tropicalization contains a subgraph of Betti number one. This does not happen for all families of elliptic curves: it can only happen for families that satisfy the technical condition $v(j(E))<0$, where $j(E)$ is the $j$-invariant associated to the family of elliptic curves. Our goal in this paper is to modify a given family of elliptic curves such that its tropicalization contains a cycle of length $-v(j(E))$. We call this a {\it{numerically faithful tropicalization}}. We  now state our main result: 
\begin{theorem}\label{MainTheorem}
Let $E$ be an elliptic curve over $K$ with $v(j(E))<0$, where $j(E)$ is the $j$-invariant of $E$. Then there exists an embedding $E\rightarrow{\mathbb{P}^{2}}$ such that its tropicalization contains a cycle of length $-v(j(E))$. 
\end{theorem}
\begin{figure}[h!]
\centering
\includegraphics[width=7.5cm]{{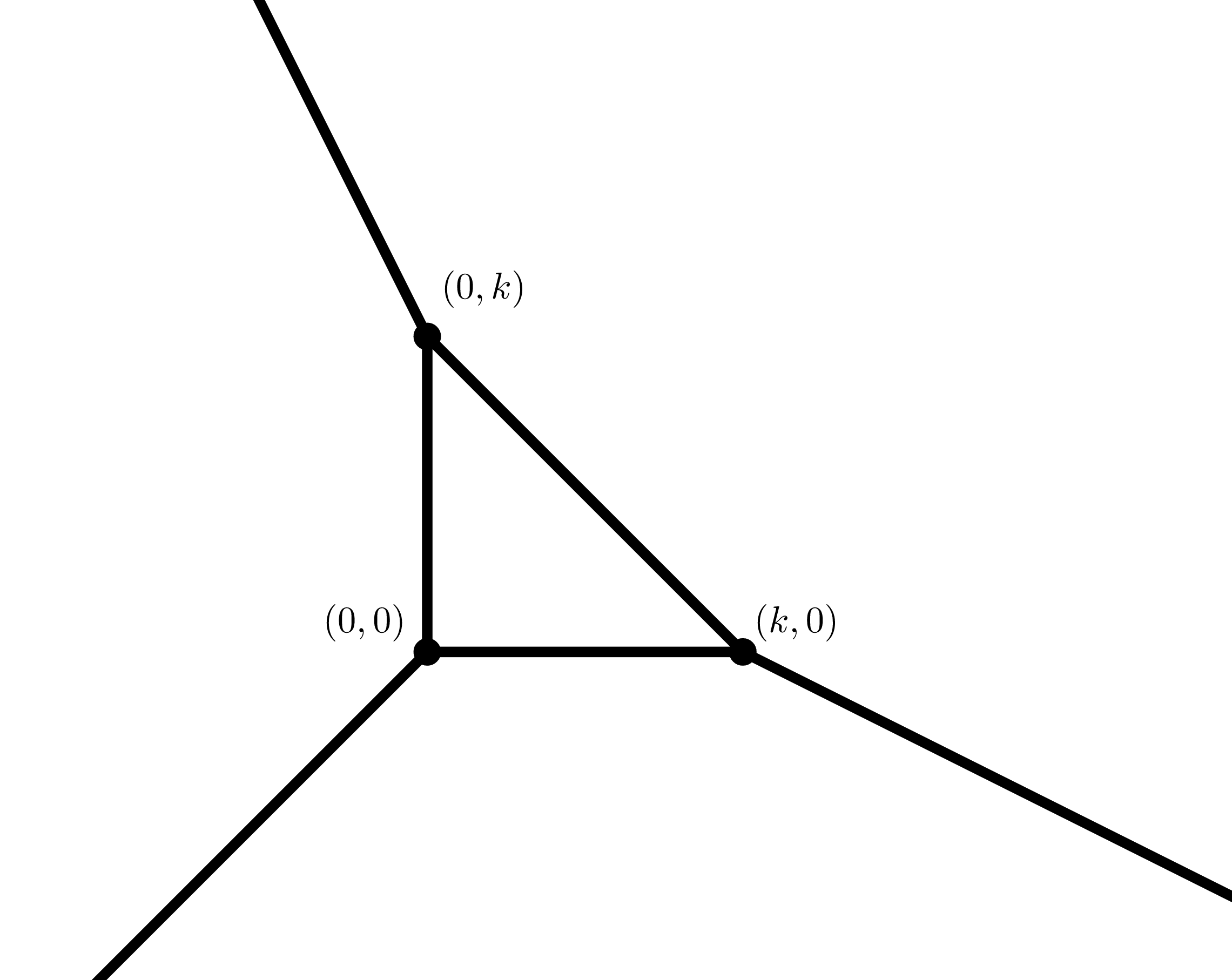}}
\caption{\label{Plaatje1} The tropicalizations obtained in Theorem \ref{MainTheorem} from the elliptic curves in Equation \ref{MainEquation}.  }%
\end{figure} 

The field $K$ in the theorem can be more general than the field of Puiseux series, but the reader can assume throughout this paper that it is the field of Puiseux series as above. To be precise: we assume that our field $K$ is an algebraically closed non-archimedean field with valuation ring $R$, maximal ideal $\mathfrak{m}$, residue field $k$ and valuation $v:K^{*}\rightarrow{\mathbb{R}}$. 
We assume furthermore throughout this paper that $\mathrm{char}(k)\neq{2,3}$.  %In fact, we will prove this theorem over more general fields over which the same tropicalization process can be defined, namely over {\it{nonarchimedean fields}} $K$. The reader can assume throughout this paper that the field $K$ is just the field of Puiseux series $\mathcal{P}$ as above. More generally, we assume that $K$ is an algebraically closed non-archimedean field with valuation ring $R$, maximal ideal $\mathfrak{m}$, residue field $k$ and valuation $v:K^{*}\rightarrow{\mathbb{R}}$. 
%We assume furhtermore throughout this paper that $\mathrm{char}(k)\neq{2,3}$.
%The previous process can be done for more general fields, namely for so-called {\it{nonarchimedean fields}} $K$. The reader can assume throughout this paper that $K$ is just the field of Puiseux series $\mathcal{P}$ as above. More generally, we assume that $K$ is an algebraically closed non-archimedean field with valuation ring $R$, maximal ideal $\mathfrak{m}$, residue field $k$ and valuation $v:K^{*}\rightarrow{\mathbb{R}}$. 
%We assume throughout this paper that $\mathrm{char}(k)\neq{2,3}$.

Several proofs of Theorem \ref{MainTheorem} have already been given, for instance \cite[Proposition 2.1]{CSHoneycomb} and \cite[Theorem 6.2]{BPR11}. Our proof will follow the strategy of \cite[Theorem 6.2]{BPR11}, but we will not use the analytic slope formula, nor any of the results on faithful tropicalizations. The idea is to use reduction theory for elliptic curves in Weierstrass form $y^2=x^3+Ax+B$ to explicitly give a pair $(f,g)\in{(K(E))^{2}}$ such that the corresponding tropicalization has the desired cycle. To find these elements, we will use inflection points on families of elliptic curves, see Figure \ref{PlaatjeDegTrop3} for an explicit example over the real numbers. Under an appropriate choice of $f$ and $g$, the affine equation that cuts out the image of $E$ in $\mathbb{P}^{2}$ is then given by %the embedded curve will be of the form
\begin{equation}\label{MainEquation}
f^2g+2a'fg-fg^2-2a'b=0
\end{equation}
for $a'$ and $b$ satisfying $v(a')=0$ and $v(b)>0$. The tropicalization of this curve then contains a triangle with sides of length $v(b)$ as in Figure \ref{Plaatje1} and we show that the valuation of the $j$-invariant of $E$ is $-3v(b)$.
 % for which it is very easy to calculate the tropicalization. 

Since reduction theory is usually only given in the discrete case (see \cite[Chapter VII]{Silv1}), we give a more or less full treatment for the non-discrete case. We define minimal models, reduction types and we show that any elliptic curve has multiplicative reduction if and only if $v(j(E))<0$. These tools don't seem to be as well-known  %are well-known in arithmetic algebraic geometry, yet they seem to be more or less unknown
 in tropical geometry as in arithmetic algebraic geometry due to the advent of Tate uniformizations, Berkovich spaces and (formal) semistable models. We chose to include most of the classical results on reduction theory (albeit in an altered form) in this text. We also included an introduction to tropical geometry in Section \ref{Tropicalizations}, giving most of the definitions, the fundamental theorem of tropical geometry and the structure theorem. 

We now give a quick review of the proofs of Theorem \ref{MainTheorem} given in \cite{CSHoneycomb} and \cite{BPR11}. In \cite{CSHoneycomb}, one starts with a two parameter family of elliptic curves whose Newton complexes are unimodular triangulations, the so-called elliptic curves in symmetric honeycomb form. More explicitly, these are given by 
\begin{equation}\label{CurveSturm}
a\cdot{(x^3+y^3+z^3)}+b\cdot{}(x^2y+x^2z+y^2x+y^2z+z^2y+z^2x)+xyz,
\end{equation}
where $a,b\in{K}$ satisfy 
\begin{equation}
{v}(a)>2\cdot{{v}(b)}>0. 
\end{equation}

For any $a,j\in{K}$ with ${v}(j)<0$ and $v(a)+v(j)>0$, they then show that there are exactly $6$ values of $b$ such that the $j$-invariant of the genus $1$ curve $E_{a,b}$ in Equation \ref{CurveSturm} is equal to $j$. Since unimodular triangulations automatically induce faithful tropicalizations by \cite[Corollary 5.28]{BPR11}, the theorem is then proved. One can also skip this theorem on faithful tropicalizations using the results in \cite{katz_markwig_markwig_2009}. There it is shown computationally that any curve in symmetric honeycomb form automatically has a cycle of length $-v(j(E_{a,b}))$. This then also directly gives the theorem. %Since this

In the proof of \cite[Theorem 6.2]{BPR11}, they start with the {\it{Tate uniformization theorem}}, which can be stated as follows. % This classical result is as follows.
 Let $E$ be an elliptic curve over a %This says that for any elliptic curve $E$ over a 
nonarchimedean, complete, algebraically closed field $K$ with $v(j(E))<0$. Then there exists an isomorphism
\begin{equation}
E(K)\simeq{K^{*}/\langle{q}\rangle}
\end{equation}
for some $q\in{K}$ with $v(q)=-v(j(E))$. They then set out to find functions $f$ and $g$ whose associated piecewise linear functions $-\mathrm{log}|f|$ and $-\mathrm{log}|g|$ induce an isometric embedding $\Sigma\rightarrow{\mathbb{R}^{2}}$ of the minimal skeleton $\Sigma$ (in the associated Berkovich space of $E$) into $\mathbb{R}^{2}$. To ensure that this is an isometric embedding, they need that at least one of the piecewise linear functions has slope equal to $1$ on every edge. They then create these functions $f$ and $g$ using the image of the three-torsion point $P=q^{1/3}$ in $E(K)$.

Our approach mostly follows the strategy outlined in the previous paragraph. There are three differences however. First of all, we do not use the Tate uniformization theorem to obtain the desired three-torsion point. Instead, we will use the geometric interpretation of three-torsion points being {\it{inflection points}} (see Lemma \ref{InflectionPoint}). To find the correct analogue of $q^{1/3}$, % which are the same as {\it{inflection points}} for elliptic curves (see Lemma \ref{InflectionPoint}). Instead, 
we work out a reduction theory for elliptic curves over $K$ and we use some classical results from arithmetic geometry to find the desired inflection point $P$. In doing this, we obtain an {\it{explicit family}} of elliptic curves $E$ with $v(j(E))<0$ and a marked inflection point $P$ reducing to the singular point, see Lemmas \ref{ExplicitFamily1} and \ref{LemmaValuation}. Moreover, other ingredients in the proof of \cite[Theorem 6.2]{BPR11} are also made explicit, as we give a pair of functions $(f,g)\in{(K(E))}^{2}$ which induce a closed embedding, and we explicitly calculate its image. % which we will explicitly calculate. %explicitly calculate the image in $\mathbb{P}^{2}$. 
The corresponding tropicalization then contains a triangle and we show that its length is equal to $-v(j(E))$. This then also highlights the second key difference: we do not use any analytic material such as the {\it{Poincar\'{e}-Lelong formula}} (See \cite[Theorem 5.15]{BPRa1}) or the criterion for faithful tropicalizations (see \cite[Corollary 5.28]{BPR11}) to abstractly conclude that our tropicalizations contain a cycle of the right length. We also note that we don't assume that our nonarchimedean field is complete. This allows us to directly use our results on the field of Puiseux series $\mathcal{P}$ for instance. 

In our Main Theorem, we show that for any elliptic curve $E/K$ with $v(j(E))<0$, we can find a tropicalization such that it contains a cycle of length $-v(j(E))$. We call this a numerically faithful tropicalization, see Section \ref{NumericallyFaithful}. This is not the same as a faithful tropicalization however. The difference is subtle and in Section \ref{FaithfulTropSection} we explore this difference. Simply put, by expanding some of the edges and contracting others, we can obtain a cycle that has the right length, but which does not faithfully represent the minimal skeleton of the associated Berkovich space since it is not injective on the contracted parts. In Section \ref{FaithfulTropSection}, we abstractly prove the existence of a numerically faithful tropicalization that is not faithful, see Example \ref{NonFaithfulExample}.

Over the last couple of years, various new results on faithful tropicalizations have been published. For instance, in \cite{SkeletonJacobian} they proved  using tropical Jacobians that for any curve $C$, there exists a rational map $C\rightarrow{\mathbb{P}^{3}}$ such that its tropicalization is an isometry onto its image. For genus two and Mumford curves, there are other more specific results in this direction, see \cite{Wagner2015} or \cite{JellMumfordFaithful}. On the other hand, one can also consider faithful tropicalizations of more general algebraic varieties. For results in these directions, we would like to direct the reader to \cite{Cueto2014} for faithful tropicalizations of the Grassmannian of planes and \cite{HypertoricFaithful} for faithful tropicalizations of hypertoric varieties. %, so the field is not just restricted to curves as in this paper.  

% and abstractly prove the existence of a numerically faithful tropicalization that is not faithful  

%Lastly, we note one other difference, namely that we don't 

The paper is structured as follows. We start by summarizing some well-known results in tropical geometry in Section \ref{Tropicalizations}. We cover the tropical semiring, tropical varieties, the fundamental theorem (Theorem \ref{FundTrop}), the structure theorem (Theorem \ref{TropComplex}) and we give the definition of a {\it{numerically faithful tropicalization}}. We then compare this definition to that of a faithful tropicalization in Section %e definition of a numerically faithful tropicalization to that of a faithful tropicalization in Section 
\ref{FaithfulTropSection}. Here we give two examples that are numerically faithful but not faithful.  In Section \ref{ModelsSection}, we introduce a reduction theory for elliptic curves over $K$ as in \cite[Chapter VII]{Silv1} and in Section \ref{FinalSection} we prove Theorem \ref{MainTheorem}. % in Section \ref{FinalSection}. 

We tried to keep the text as elementary as possible, giving examples of the notions introduced wherever possible. As such, we believe that this paper serves as a didactic tool in understanding some of the more  abstract material in \cite{BPR11} and \cite{BPRa1} in concrete terms. To further aid the reader in this, we will % by casting it in this more elementary form. 
point out any similarities and differences between our approach and the one in \cite{BPR11} as we come across them.

We will be using most of the conventions regarding algebraic geometry as introduced in \cite[Chapters II and III]{Silv1}. %We will also be using \cite[Chapter VII]{Silv1} for reduction theory, albeit over non-Noetherian valuation rings. 
For tropical geometry, we will mostly be using \cite[Chapter 3]{tropicalbook}. We also refer the reader to those two books for more background information regarding these topics.

\section*{Acknowledgements}

The author would like to thank Pedro \'{A}ngel Castillejo Blasco for his comments and suggestions on an early version of this paper. The author would also like to thank the reviewers for their comments and suggestions, especially for suggesting the notion of a {\it{numerically faithful tropicalization}}, see Sections \ref{NumericallyFaithful} and \ref{FaithfulTropSection}.

\section{Tropicalizations}\label{Tropicalizations}

In this section we discuss the notion of a tropicalization of a closed variety $X/K$ inside $(K^{*})^{n}$, where $K$ is a nonarchimedean field as in Section \ref{Introduction}. %a torus $\mathbb{G}_{m}^{n}$. %  map $X\rightarrow{\mathbb{P}^{n}}$ over $K$, where $X$ is a variety over $K$. 
We will be mostly interested in the case where $X$ is induced by an algebraic curve such as an elliptic curve over $K$. We introduce the tropical semiring and we recall the fundamental theorem of tropical geometry, which we will use in our main theorem as an easy tool to calculate tropicalizations. We refer the reader to \cite{tropicalbook} for more background information regarding this topic. In Section \ref{NumericallyFaithful}, we study the underlying combinatorial structure of the tropicalization of a variety and we discuss the tropical structure theorem: Theorem \ref{TropComplex}. This in turn allows us to define the lattice length on the tropicalization of a curve, see Definition \ref{LatticeLength}. After this we define a numerically faithful tropicalization for the special case of elliptic curves in Definition \ref{NumericallyFaithfulDefinition}. In Section \ref{FaithfulTropSection} we then compare this definition to the definition of a faithful tropicalization given in \cite[Section 5.15]{BPR11} and we give two examples of a numerically faithful tropicalization that is not faithful. 

Consider the extended real line $\mathcal{R}:=\mathbb{R}\cup\{\infty\}$ with its natural total order. We turn this into a semiring by defining the following two operations on $\mathcal{R}$:
\begin{align*}
a\oplus{b}&=\mathrm{min}\{a,b\},\\
a\odot{b}&=a+b.
\end{align*}
Here we set $\infty\,{\odot{\,b}}=b\odot{\infty}=\infty$ for any $b\in\mathcal{R}$.  %and 
%See [Sturmfels,Maclagan] for more on these operations.
 We note that these operations mimic the following two identities of the valuation function $v:K\rightarrow{\mathcal{R}}$:
\begin{align*}
v(a+b)&\geq{\mathrm{min}\{v(a),v(b)\}},\\
v(a\cdot{b})&=v(a)+v(b).
\end{align*}
A multivariate tropical polynomial in $n$ variables ${\bf{x}}=(x_{1},x_{2},...,x_{n})$, written as $f({\bf{x}})=\bigoplus_{i\in{I}}^{}a_{i}\odot{{\bf{x}}^{i}}$ for $a_{i}\in{\mathbb{R}}$, is then the function $\mathcal{R}^{n}\rightarrow{\mathcal{R}}$ given by
\begin{equation}\label{TropicalPolynomial}
{\mathbf{x}}\mapsto{\bigoplus_{i\in{I}}^{}a_{i}\odot{{\bf{x}}^{i}}}.
\end{equation}
Here $I$ is a finite subset of $\mathbb{N}^{n}$, similar to the case of multivariate polynomials $K[x_{1},x_{2},...,x_{n}]$. Any monomial of the form $a_{i}\odot{{\bf{x}}^{i}}$ is referred to as a term of $f$. 
\begin{exa}\label{ExampleTropicalPolynomial}
Consider the tropical polynomial $f=(1\odot{}x\odot{y})\oplus{}(2\odot{2x})={\mathrm{min}\{1+x+y,2+2x\}}$.   This defines the following piecewise linear function on $\mathbb{R}^{2}$:
\begin{equation}
f(x,y)=\left\{\begin{array}{rl}
2+2x  & \text{ for } y\leq{}1+x,\\
1+x+y &  \text{ for } y>1+x.
\end{array}\right.
\end{equation}% for $y>1+x$ we have that $f=2+2x$, for $y<1+x$ we have that  
\end{exa} 

Using a slight alteration of the definition of a tropical polynomial, we obtain the notion of a tropical Laurent polynomial. The definition imitates that of the ring $K[x_{1},x_{1}^{-1},...,x_{n},x_{n}^{-1}]$. A tropical Laurent polynomial is by definition a piecewise linear function as in Equation \ref{TropicalPolynomial}, where we now allow the index set $I$ to be a subset of $\mathbb{Z}^{n}$. For instance, for tropical Laurent polynomials in one variable $x$, we have that $x^{-i}$ (written tropically) is equal to the function $-i\cdot{x}$. The multivariate case is similar. % for a single variable $x$ and similarly for multivariate tropical polynomials. 

We now define the tropicalization of an algebraic set $V(I)\subset{(K^{*})^{n}}$ corresponding to an ideal $I\subset{K[x_{1},x_{1}^{-1},...,x_{n},x_{n}^{-1}]=:K[{\bf{x}},{\bf{x}}^{-1}]}$. To that end, we first define the tropicalization of a multivariate polynomial $f\in{K}[{\bf{x}},{\bf{x}}^{-1}]$. Let $f=\sum_{i\in{I}}a_{i}{\bf{x}}^{i}$ be a multivariate Laurent polynomial. We define the tropicalization of $f$ to be the tropical Laurent polynomial given by
\begin{equation}
\mathrm{trop}(f)=\bigoplus_{i\in{I}}v(a_{i})\odot{{\bf{x}}^{i}},
\end{equation}
where the product ${\bf{x}}^{i}$ is now a tropical product.
\begin{exa}
Let $f=\varpi^2+\varpi{x^2}+y^3+x^{-1}$, where $\varpi\in{K}$ satisfies $v(\varpi)=1$. Then $\mathrm{trop}(f)=\mathrm{min}\{2,1+2x,3y,-x\}$. 
\end{exa}

For any multivariate Laurent polynomial $f$, we now introduce the notions of a tropical hypersurface and a tropical variety.
\begin{mydef}\label{DefinitionTropicalHypersurface}
Let $f$ be a multivariate Laurent polynomial with monomial terms $h_{i}$. We define the {\bf{tropical hypersurface}} corresponding to a $f$ to be the set of points ${\bf{x}}\in\mathbb{R}^{n}$ such that $\mathrm{trop}(f)({\bf{x}})=\mathrm{trop}(h_{i})({\bf{x}})=\mathrm{trop}(h_{j})({\bf{x}})$ for at least two different terms of $f$. It is denoted by $\mathrm{trop}(V(f))$. A {\bf{tropical pre-variety}} is then an intersection of these tropical hypersurfaces and a {\bf{tropical variety}} is a subset of $\mathbb{R}^{n}$ of the form
\begin{equation}
\mathrm{trop}(V(I))=\bigcap_{f\in{I}}\mathrm{trop}(V(f)),
\end{equation} 
where $I$ is an ideal in the Laurent polynomial ring $K[x_{1},x_{1}^{-1},...,x_{n},x_{n}^{-1}]$. We refer to this subset $\mathrm{trop}(V(I))$ as the {\it{tropicalization}} of $I$. 
\end{mydef}

  %We define the tropicalization of an ideal $I\subset{K[x_{1},x_{1}^{-1},...,x_{n},x_{n}^{-1}]}$ to be
%\begin{equation}
%\mathrm{trop}(V(I))=\bigcap_{f\in{I}}\mathrm{trop}(V(f)).
%\end{equation} % in $\mathbb{R}^{n}$.
%\end{mydef}
\begin{exa}
Let $f$ be the tropical polynomial from Example \ref{ExampleTropicalPolynomial}. Then the tropical hypersurface corresponding to $f$ is given by $\mathrm{trop}(V(f))=\{(x,y)\in\mathbb{R}^{2}:y=1+x\}$. 
\end{exa}

\begin{exa}\label{Prevarieties1}
Not every tropical pre-variety is a tropical variety, as we will see in Example \ref{Prevarieties2}. The idea is to take a maximal ideal $I$ inside $K[x,x^{-1},y,y^{-1}]$ with two generators $f_{1}$ and $f_{2}$ such that the intersection of the two corresponding tropical hypersurfaces is a half-ray. This intersection is one-dimensional, whereas the zero set of the ideal $I$ is zero-dimensional. We will see that this causes any ideal $J$ with $\mathrm{trop}(V(J))=\mathrm{trop}(V(f_{1}))\cap{\mathrm{trop}(V(f_{2}))}$ to be zero-dimensional, which contradicts the fundamental theorem of tropical geometry: Theorem \ref{FundTrop}.    %The fundamental theorem of tropical geometry (Theorem [...]) then easily shows that this intersection   %We will show that any ideal $J$ such that $\mathrm{trop}(V(J))=\mathrm{trop}(f_{1})\cap{\mathrm{trop}(f_{2})}$ must be zero-dimensional, which will contradict the fundamental theorem  %We will show that any ideal that would give rise to this intersection  
The details are given in Example \ref{Prevarieties2}.  % The $f_{i}$ are then chosen so that the intersection of the two tropical hypersurfaces corresponding the $f_{i}$ is a half-ray. This half-ray %see \cite[Example 3.2.2]{tropicalbook}. We will give a proof of this after Theorem \ref{FundTrop}. 
\end{exa}

% We will follow the example given in \cite[Example 3.2.2]{tropicalbook}. Consider the ideal $I=(x+y+1,y+2x)\subset{K[x,x^{-1},y,y^{-1}]}$, where $K$ is a valued field. The tropical hypersurface $\mathrm{trop}(x+y+1)$ consists of three line segments and the tropical hypersurface $\mathrm{trop}{(y+2x)}$ consists of a single line segment. They intersect in the half-ray given by $\{(x,y)\in\mathbb{R}^{2}:x=y<0\}$.     %$K=\mathbb{C}((t))$.%For instance, the intersection of the two tropical hypersurfaces  
%\end{exa}
%\begin{mydef}
%Let $\phi:C\rightarrow{\mathbb{A}^{n}}$ be a closed embedding and consider the 
%\end{mydef}
%Suppose now that we have an ordinary multivariate polynomial $f\in{K[x_{1},...,x_{n}]}$. 

%We can now also define the tropicalization of an ideal $I$ in $K[x_{1},x_{1}^{-1},...,x_{n},x_{n}^{-1}]$ and its corresponding algebraic set $V(I)\subset{(K^{*})^{n}}$.  

We now relate the construction of tropical varieties to another, perhaps more natural, construction. Consider the "naive" tropicalization map
\begin{align}\label{NaiveMap}
\mathrm{val}:(K^{*})^{n}&\rightarrow{\mathbb{R}^{n}}\\
(x_{1},x_{2},...,x_{n})&\mapsto{(v(x_{1}),v(x_{2}),...,v(x_{n}))}.
\end{align}
\begin{mydef}\label{NaiveTropicalization}
Let $I$ be an ideal of the Laurent polynomial ring $K[\mathbf{x},\mathbf{x}^{-1}]$. The {\it{naive tropicalization}} of an algebraic set $V(I)$ is defined to be the closure of the image of $V(I)$ under $\mathrm{val}(\cdot{})$. We denote it by $\mathrm{val}(V(I))$. It is also referred to as the {\bf{non-archimedean amoeba}} associated to $I$, see \cite[Definition 1.1.1]{Einsiedler2006}.
 \end{mydef}
\begin{exa}\label{PlaneCurve}
Let $C$ be the plane curve defined by $f:=\varpi^{-3}x_{1}+\varpi^{-2}x_{2}-1=0$. We consider three cases.
\begin{itemize}
\item Suppose that $v(x_{1})>3$. Then we must have $v(x_{2})=2$ by considering the valuations of both sides of the equation $\varpi^{-3}x_{1}=1-\varpi^{-2}x_{2}$. 
\item Suppose that $v(x_{2})>2$. Then similarly $v(x_{1})=3$.
\item Suppose now that $v(x_{1})\leq{3}$ and $v(x_{2})\leq{2}$. Then the valuations of $\varpi^{-3}x_{1}$ and $\varpi^{-2}x_{2}$ have to be equal in order to obtain $v(1)=0$. In other words, $-3+v(x_{1})=-2+v(x_{2})$, which gives $v(x_{1})=1+v(x_{2})$. 
\end{itemize}
\begin{figure}[h!]
\centering
\includegraphics[width=5cm]{{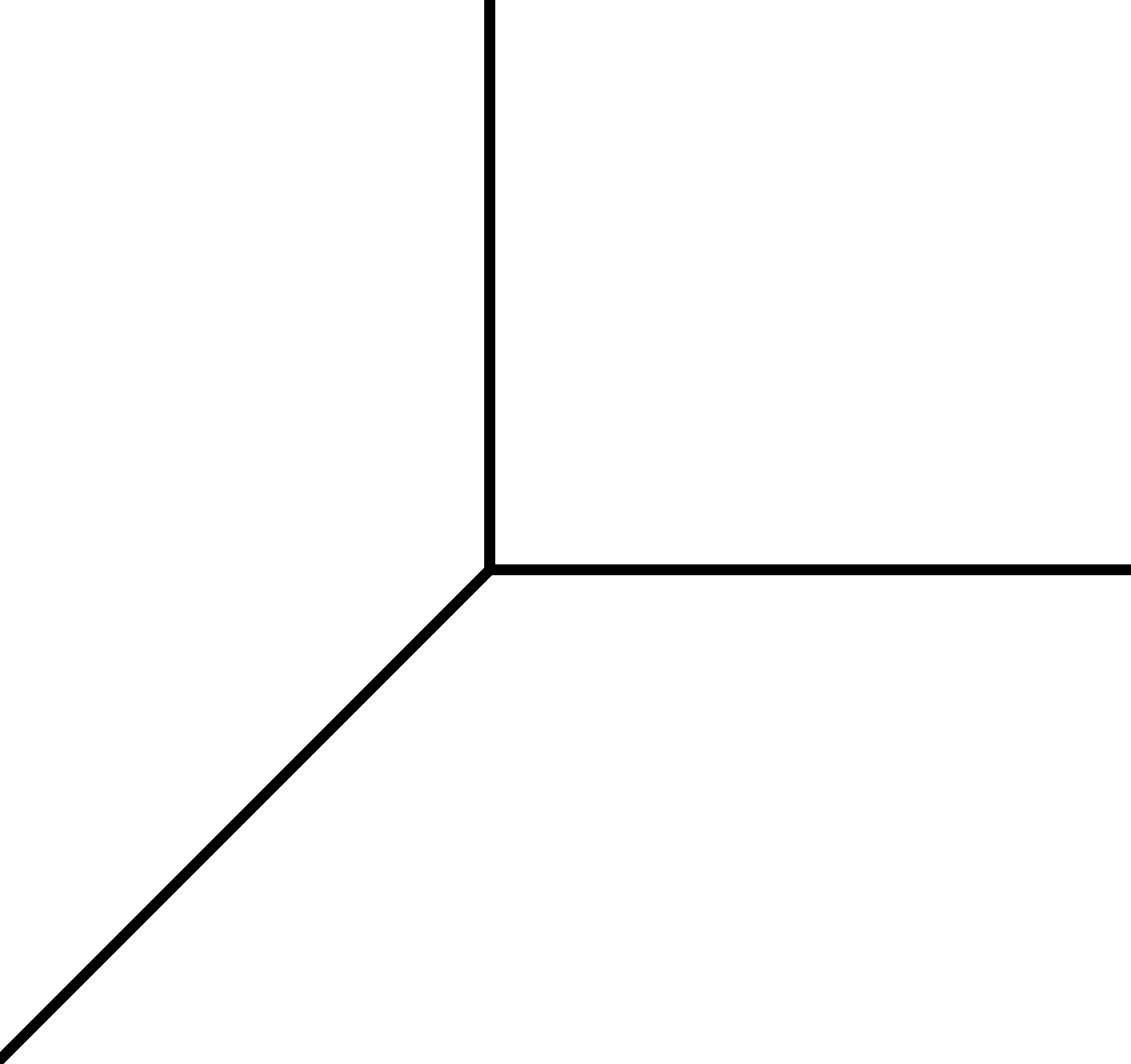}}
\caption{\label{Plaatje2} The tropicalization obtained in Example \ref{PlaneCurve}. }\end{figure}
We thus see that the nonarchimedean amoeba consists of three linear pieces, as depicted in Figure \ref{Plaatje2}.
\end{exa} 

The good news now is that the naive tropicalization of $V(I)$ coincides with the tropicalization defined in Definition \ref{DefinitionTropicalHypersurface}. This is also known as the fundamental theorem of tropical geometry.
\begin{theorem}\label{FundTrop}{\bf{(Fundamental theorem of tropical geometry)}}

Let $\mathrm{trop}(V(I))$ be the tropical variety defined in Definition \ref{DefinitionTropicalHypersurface} and let $\mathrm{val}(V(I))$ be the naive tropicalization defined in Definition \ref{NaiveTropicalization}. Then 
\begin{equation}
\mathrm{trop}(V(I))=\mathrm{val}(V(I)).
\end{equation}
\end{theorem}
\begin{proof}
See \cite[Theorem 3.2.3]{tropicalbook}.
\end{proof}
\begin{rem}\label{RadicalTropicalization}
Since the algebraic variety $V(I)$ only depends on the radical $\sqrt{I}$ of $I$ in the sense that $V(I)=V(\sqrt{I})$, we conclude by Theorem \ref{FundTrop} that $\mathrm{trop}(V(I))=\mathrm{trop}(V(\sqrt{I}))$. 
\end{rem}

\begin{exa}\label{TropicalizationExample1}
Let $f=\varpi^{-3}x_{1}+\varpi^{-2}x_{2}-1$ as in Example \ref{PlaneCurve}. Its tropical polynomial is then given by
\begin{equation}
\mathrm{trop}(f)=\mathrm{min}\{x_{1}-3,x_{2}-2,0\}.
\end{equation}
We then easily see that the points $(x_{1},x_{2})$ where the minimum is attained at least twice is exactly equal to the non-archimedean amoeba calculated in Example \ref{PlaneCurve}. 
\end{exa}

\begin{exa}\label{Prevarieties2}
We now continue the train of thought in Example \ref{Prevarieties1} and give an example of a tropical pre-variety $S$ that is not a tropical variety. Our proof that $S$ is not a tropical variety will use Theorem \ref{FundTrop} together with some commutative algebra.  %The example is the same as in %\cite{tropicalbook} %Using some commutative algebra and Theorem \ref{FundTrop}, we will give a rather strict proof that it Our proof will based on commutative algebra and Theorem \ref{FundTrop}. 
\begin{figure}[h!]
\centering
\includegraphics[width=6cm]{{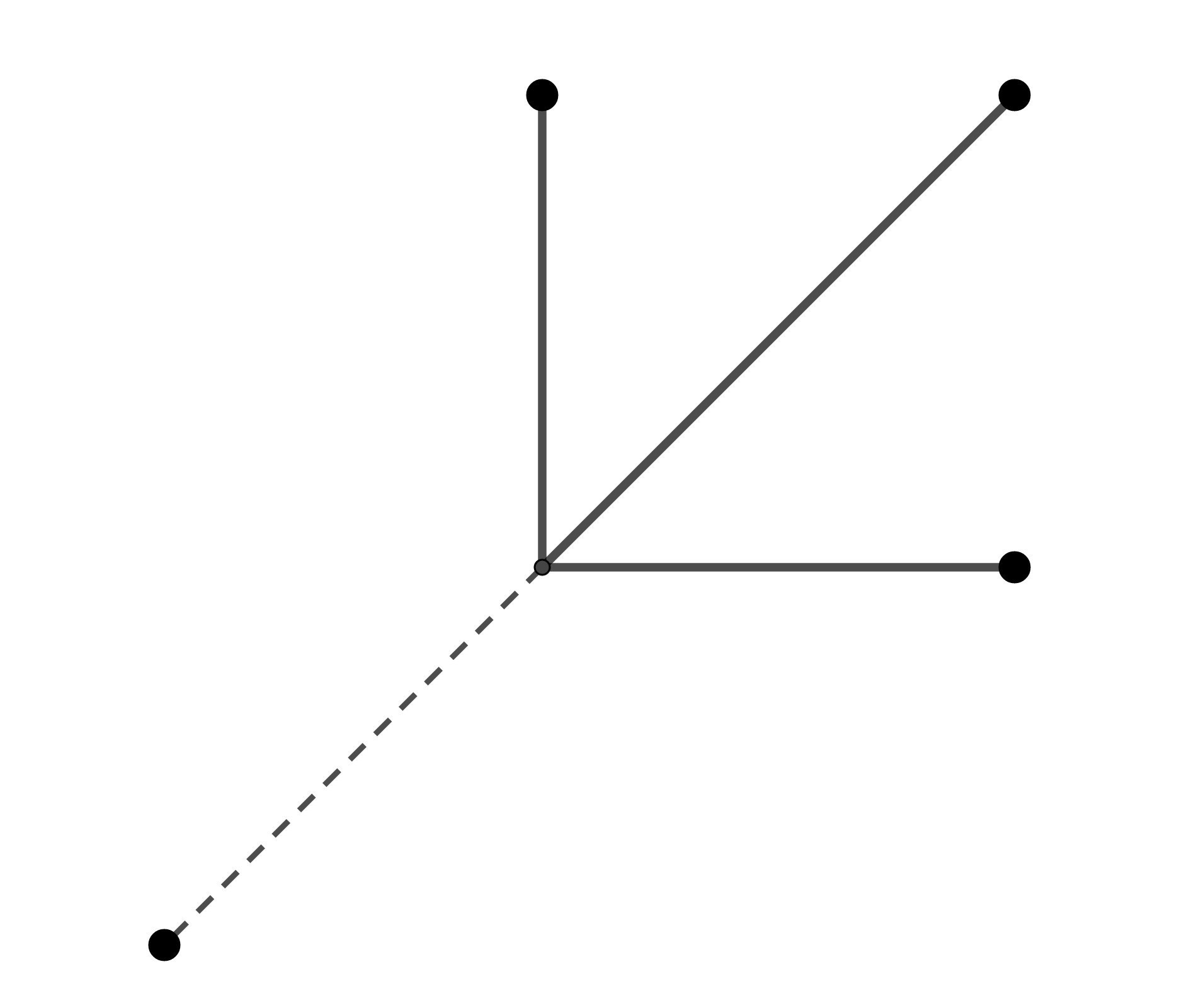}}
\caption{\label{Plaatje691} The dotted line gives a tropical pre-variety which is not a tropical variety, see %The intersection of two tropical varieties in 
Example \ref{Prevarieties2} for the proof.} %It is an intersection of which itself is only a tropical pre-variety. }%The tropicalization obtained in Example \ref{PlaneCurve}. }
\end{figure}

Consider the ideal $I=(1+x+y,x+2y)$ in $K[x,x^{-1},y,y^{-1}]$. %, where $K$ is a valued field. 
One then easily checks that $V(I)=\{(-1,2)\}$, so Theorem \ref{FundTrop} tells us that $\mathrm{trop}(V(I))$ consists of a single point. We now consider the tropicalizations of the hypersurfaces defined by $f_{1}=1+x+y$ and $f_{2}=x+2y$. The tropicalization of $V(1+x+y)$ is then as in Examples \ref{PlaneCurve} and \ref{TropicalizationExample1} and the tropicalization of $V(x+2y)$ is a single line segment. They intersect in the half-ray $S=\{(w_{0},w_{1}):w_{0}=w_{1}<0\}\subset{\mathbb{R}^{2}}$, see Figure \ref{Plaatje691}. We claim that this half-ray is not of the form $\mathrm{trop}(V(J))$ for any ideal $J$ in $K[x,x^{-1},y,y^{-1}]$. %\mathbf{x}^{-1}]$. 

Suppose for a contradiction that we can write $S=\mathrm{trop}(V(J))$ for some ideal $J$. By Remark \ref{RadicalTropicalization}, we can and do assume that $J$ is radical. By \cite[Corollary 2.12]{Eisenbud1995} we find that $J$ is the intersection of all prime ideals that contain $J$ and thus it is the intersection of the finitely many minimal prime ideals that contain it (see \cite[Chapter 3]{Eisenbud1995} for the notion of a minimal prime ideal). %Consider the primary decomposition $\bigcap{\mathfrak{p}_{i}}=J$ as in \cite[Chapter 3]{Eisenbud1995} and \cite[Chapter 3]{Atiyah1994}. Here the $\mathfrak{p}_{i}$ are prime ideals by our assumption on $J$. 
%\cap{{\mathfrak{q}_{i}}}=J$, where the radical of $\mathfrak{q}_{i}$ is $\mathfrak{p}_{i}$. These prime ideals correspond to the irreducible components of $J$, see [...]. 
We will show that these prime ideals all have height $2$, which shows that they are maximal since the dimension of $K[x,x^{-1},y,y^{-1}]$ is two. %(as it is a localization of a domain of .  

Note that none of these prime ideals can have height zero, since $J\neq{0}$. %this would imply that the tropicalization of $J$ is the entire plane $\mathbb{R}^{2}$, which it is not.
 Suppose on the other hand that the height of such a $\mathfrak{p}_{i}$ is one. Then $\mathfrak{p}_{i}$ is principal, since $K[x,x^{-1},y,y^{-1}]$ is a unique factorization domain (being a nonzero localization of a unique factorization domain). We can thus write $\mathfrak{p}_{i}=(f_{i})$ for an irreducible element $f_{i}\in{K[x,x^{-1},y,y^{-1}]}$. Since $\mathfrak{p}_{i}=(f_{i})\supset{J}$, we obtain $V(f_{i})\subset{V(J)}$ and thus $\mathrm{trop}(V(f_{i}))\subset{\mathrm{trop}(V(J))}=S$. But this can't happen: $\mathrm{trop}(V(f_{i}))$ will contain points with positive valuation. Indeed, for almost all $x_{0}$ (i.e. a Zariski dense set in $K$\footnote{To be explicit, we can write $f=\sum{c_{i}(x)y^{i}}$ where the $c_{i}(x)$ are elements of $K[x,x^{-1}]$ and consider a nonzero coefficient $c_{i}(x)$ for $i\neq{0}$. The set $U=K\backslash{}V(c_{i}(x))$ inside $K$ with its Zariski topology then has the required properties.}), we have that $f_{i}(x_{0},y)$ is a nonconstant polynomial. Since $K$ is algebraically closed, there exists a $y_{0}$ such that $f_{i}(x_{0},y_{0})=0$ and for any such $y_{0}$,  %since $K$ is algebraically closed. %\footnote{For those who know something about schemes: we're saying that the map to $\mathrm{Spec}(K[x,x^{-1}])$ induced by $f$ is quasi-finite on a dense open subset of $\mathrm{Spec}(K[x,x^{-1}])$.}. 
%For any such $y_{0}$, 
we then have that $P=(x_{0},y_{0})\in{V(f_{i})}\subset{V(J)}$. If we now chose $x_{0}$ such that $v(x_{0})>0$, this will contradict $\mathrm{trop}(V(f_{i}))\subset{\mathrm{trop}(V(J))}$. Indeed, Theorem \ref{FundTrop} gives us $\mathrm{val}(P)\in\mathrm{val}(V(f_{i}))=\mathrm{trop}(V(f_{i}))\subset{S}$, which is the desired contradiction. We thus conclude that the minimal primes of $J$ are maximal ideals. That is, $J$ defines a zero-dimensional variety. But then $S=\mathrm{trop}(V(J))=\mathrm{val}(V(J))$ must also consist of finitely many points, another contradiction. We conclude that $S$ is not a tropical variety.  %a contradiction.  % Let us prove this algebraically
    %Then $J\neq{(0)}$ because $S$ is not the entire plane. %So there exists an $f\in{J}\backslash\{0\}$. 
%For any minimal associated prime $\mathfrak{p}$ of $J$, we then find $V(\mathfrak{p})\subset{V(J)}$ and thus $\mathrm{trop}(V(\mathfrak{p}))\subset{\mathrm{trop}(V(J))}$. We will show that any minimal prime of $\mathrm{ht}(\mathfrak{p})=1$ leads to a contradiction. Since $K[x,x^{-1},y,y^{-1}]$ is a unique factorization domain, any such $\mathfrak{p}$ is of the form $\mathfrak{p}=(f)$ for some irreducible $f$. 

%Consider the tropicalization of this $f$. We claim that it contains points in the positive half-space $\{(w_{0},w_{1}):w_{0}>0\}$. Let us prove this algebraically. We write $f=\sum{i=-k}^{n}c_{i}(x)y^{i}$, where the $c_{i}(x)$ are in $K[x,x^{-1}]$. Consider the closed subset $V(c_{n}(x))$ of $K^{2}$ in terms of the Zariski topology. Consider the points $P_{i}=(it)$.     % Indeed, write $f=\sum_{i=0}^{n}c_{i}(x)y^{i}$, where $c_{i}(x)$ are polynomials  
 %, %for instance by Krull's theorem.    %Using the fundamental theorem of tropical geometry, we can easily prove that certain tropical prevarieties of $\mathbb{R}^{n}$ are not tropical varietiesLet us consider \cite[Example 3.2.2]{tropicalbook}  %That is, consider the ideal generated by 
\end{exa}

We now investigate some of the combinatorics underlying a tropical variety. To that end, we first recall some polyhedral geometry. A good reference for the material here is \cite{Ziegler1995}. We will mostly follow %We will mostly follow 
\cite[Section 2.3]{tropicalbook} for the tropical side of things. %A tropical variety is combinatorial object, as We now recall some pof the combinatorial properties of a tropical variety. We will mostly follow \cite[Chapter ...]{Ziegler1995} and \cite[Section 2.3]{tropicalbook}. 

A {\it{polyhedron}} $P$ is a subset of $\mathbb{R}^{n}$ that can be written as
\begin{equation}\label{Polyhedron}
P=\{\mathbf{x}\in\mathbb{R}^{n}:A\mathbf{x}\leq{\mathbf{b}}\},
\end{equation}
where $A$ is a $d\times{n}$-matrix and $\mathbf{b}\in\mathbb{R}^{d}$. A {\it{face}} of a polyhedron $P$ is then a subset of $P$ that can be written as
\begin{equation}
\mathrm{face}_{\mathbf{w}}(P)=\{\mathbf{x}\in{P}:\mathbf{w}(\mathbf{x})\leq\mathbf{w}(\mathbf{y}) \text{ for all }\mathbf{y}\in{P}\},
\end{equation}
where $\mathbf{w}$ is some element of the dual space $(\mathbb{R}^{n})^{\vee}:=\mathrm{Hom}(\mathbb{R}^{n},\mathbb{R})$. A {\it{facet}} of a polyhedron is a face that is not contained in a larger proper face. %A face of a polyhedron that is not contained in a larger proper face is called a {\it{facet}}.  
A {\it{polyhedral complex}} is a finite collection $\Sigma$ of polyhedra in $\mathbb{R}^{n}$ such that the following two hold:
\begin{enumerate}
\item If $P\in{\Sigma}$, then any face of $P$ is also in $\Sigma$.
\item If $P$ and $Q$ are in $\Sigma$, then $P\cap{Q}$ is either empty or a face of both $P$ and $Q$. 
\end{enumerate}
The elements of a polyhedral complex are referred to as {\it{cells}} and cells of a polyhedral complex that are not faces of any larger cell are the {\it{facets}} of the complex. 
The underlying points of a polyhedral complex $\Sigma$, or the {\it{support}} of $\Sigma$, is defined by
%The {\it{support}} of a polyhedral complex $\Sigma$ is given by
\begin{equation}
\mathrm{supp}(\Sigma)=|\Sigma|:=\{\mathbf{x}\in\mathbb{R}^{n}:\mathbf{x}\in{P} \text{ for some }P\in\Sigma\}. 
\end{equation}
For a polyhedron $P$ inside a polyhedral complex $\Sigma$, the {\it{affine span}} of $P$ is the smallest affine linear space that contains it. The {\it{dimension}} of the polyhedron $P$ is then the dimension of this affine linear space. We say that a polyhedral complex $\Sigma$ is pure of dimension $n$ if every facet of $\Sigma$ has dimension $n$. % polyhedron $P$

We now consider $\Gamma$-rational polyhedra and polyhedral complexes, where $\Gamma$ is a subgroup of the additive group of $\mathbb{R}$. For us, this $\Gamma$ will always be the value group of the discretely valued field $K$. A {\it{$\Gamma$-rational polyhedron}} is a polyhedron that is defined by a matrix $A$ with integer entries and a vector $\mathbf{b}\in{(\Gamma)^{d}}$ as in Equation \ref{Polyhedron}. Similarly, a $\Gamma$-rational polyhedral complex is a polyhedral complex consisting of $\Gamma$-rational polyhedra and its support is just its support as a polyhedral complex.  %We will be mostly interested in the case where $\Gamma$ is the {\it
%For any subgroup $\Gamma\subset{(\mathbb{R},+)}$, we now consider $\Gamma$-rational polyhedra and polyhedral complexes. A {\it{$\Gamma$-rational polyhedron}} is a polyhedron that is defined by a $\mathbb{Q}$-rational matrix $A$ and a vector $\mathbf{b}\in{(\Gamma)^{d}}$ as in Equation \ref{Polyhedron}. Similarly, a $\Gamma$-rational polyhedral complex is a polyhedral complex consisting of $\Gamma$-rational polyhedra and its support is just its support as a polyhedral complex. For us, $\Gamma$ will always be the {\it{value group}} if the discretely valued field $K$.   

\begin{theorem}\label{TropComplex}
{\bf{[Structure theorem for tropical varieties]}}
Let $I$ be an ideal of $K[\mathbf{x},\mathbf{x}^{-1}]$ defining the variety $X:=V(I)$. Let $\mathrm{trop}(X)$ be its tropicalization, as defined in Definition \ref{DefinitionTropicalHypersurface}. Then $\mathrm{trop}(X)$ is the support of a $\Gamma$-rational polyhedral complex $\Sigma$. If $X$ is an irreducible variety of dimension $n$, then the $\Gamma$-rational polyhedral complex $\Sigma$ obtained above is pure of dimension $n$. 

\end{theorem}
\begin{proof}
See \cite[Theorem 3.2.3.(2). and Lemma 3.2.10]{tropicalbook}. An early version of this theorem can also be found in \cite[Theorem A]{Bieri1984}.  %or [...] for the original versions. 
\end{proof}

\begin{exa}\label{ExplicitComplex}
Let us show that the tropicalization in Examples \ref{PlaneCurve} and \ref{TropicalizationExample1} are $\mathbb{Q}$-rational polyhedral complexes by explicitly giving the matrices $A$ and the vectors $\mathbf{b}$ as in Equation \ref{Polyhedron}. We will use $x_{1}$ and $x_{2}$ as coordinates in $\mathbb{R}^{2}$. 

We have three linear pieces $S_{i}\subset{\mathbb{R}^{2}}$.  $S_{1}$ is defined by the equations $x_{1}=3$ and $x_{2}\geq{2}$, $S_{2}$ is defined by $x_{2}=2$ and $x_{1}\geq{3}$ and $S_{3}$ is defined by $x_{2}=x_{1}-1$ and $x_{2}\leq{2}$ and $x_{1}\leq{3}$. The corresponding matrices $A_{i}$ and vectors $\mathbf{b}_{i}$ are then given by:
\begin{align}
A_{1}=&
\begin{pmatrix}
    1 & 0 \\
    -1 & 0\\
    0 & -1
\end{pmatrix},
{\mathbf{b}}_{1}=
\begin{pmatrix}
3\\
-3\\
-2
\end{pmatrix},\\
A_{2}=&
\begin{pmatrix}
    0 & 1 \\
    0 & -1\\
    -1 & 0
\end{pmatrix},
{\mathbf{b}}_{2}=
\begin{pmatrix}
2\\
-2\\
-3
\end{pmatrix},\\
A_{3}=&
\begin{pmatrix}
    -1 & 1 \\
    1 & -1\\
    0 & 1
\end{pmatrix},
{\mathbf{b}}_{3}=
\begin{pmatrix}
-1\\
1\\
2
\end{pmatrix}.
\end{align}

Note that this polyhedral complex is pure of dimension one, which makes sense in view of Theorem \ref{TropComplex} since $V(I)$ defines an irreducible curve. %their polyhedral complexes are of dimension 

 %Let us give the corresponding matrices %$S_{1}$  is defined by the equations: $x_{1}\geq{3}$ and $x_{2}=2$. We then take the matrix 

%which are defined by the equations $x_{1}
\end{exa}

\subsection{Lattice lengths and numerically faithful tropicalizations}\label{NumericallyFaithful}

In this section we study the case of tropical curves in more detail. Using the structure theorem from the last section, we define the lattice length for any segment of a tropical curve. We then specialize to elliptic curves and define the notion of a {\it{numerically faithful tropicalization}} using the $j$-invariant of the elliptic curve. In the next section, we will relate this to the notion of a {\it{faithful tropicalization}}, as introduced and studied in \cite[Section 5.15]{BPR11}.   %This notion says that the tropicalization of the elliptic curve %To calculate the tropicalizations of these curves, we recall the notion of a Newton complex, which is a polyhedral complex dual to the polyhedral complex defined by the tropical curve from Theorem \ref{TropComplex}. We then study an even more specific class of tropical varieties, namely tropical elliptic curves. For these tropical curves we define the notion of a {\it{numerically faithful}} tropicalization using the $j$-invariant and we compare this to the definition of a faithful tropicalization in \cite{BPR11}. %, which are tropicalization%narrow down our %limit ourselves even further and 

We start with the definition of a tropical curve associated to an algebraic curve. 
\begin{mydef}
{\bf{[Tropical curves]}}
Let $X\subset{(K^{*})^{n}}$ be an irreducible curve. We define the tropical curve associated to $X$ to be its tropicalization $\mathrm{trop}(X)$ as in Definition \ref{DefinitionTropicalHypersurface}. %[...]. 
\end{mydef}

By Theorem \ref{TropComplex}, we have that $\mathrm{trop}(X)$ is the support of a $\Gamma$-rational polyhedral complex, pure of dimension one. If we have a bounded edge $e$ in $\mathrm{trop}(X)$, then it belongs to some polyhedral cone $P$ defined by a matrix $A$ with values in $\mathbb{Z}$ and a vector $\mathbf{b}$ with values in $(\Gamma)^{n}$.  The directional vector $\mathbf{v}$ corresponding to $e$ then lies in the polyhedral cone $P'$ defined by $A\mathbf{x}\leq{0}$. 

We now recall the {\it{fundamental theorem of polyhedral geometry}}, also known as the Minkowski-Weyl theorem. Let $L$ be a subfield of the real numbers (such as $\mathbb{Q}$). The fundamental theorem then says that a subset of $L^{n}$ is defined by linear inequalities $A\mathbf{x}\leq{0}$ (where $A$ is a matrix with entries in $L$) if and only if it is equal to the cone generated by a finite set of vectors in $L^{n}$. Here the cone generated by a finite set of vectors $Y=\{y_{1},...,y_{k}\}\subset{L^{n}}$ is given by
\begin{equation}
\mathrm{cone}(Y)=\{\lambda_{1}y_{1}+...+\lambda_{k}y_{k}:\lambda_{i}\geq{0}\}.
\end{equation}

We direct the reader to \cite[Theorem 1.2]{Ziegler1995} for a proof of the fundamental theorem. We note that it is usually only stated and proved for the field of real numbers $\mathbb{R}$, but the result holds for any subfield $L$ of $\mathbb{R}$. Indeed, the Fourier-Motzkin method described in \cite[Section 1.2]{Ziegler1995} only uses the ordering on $L$ (and not any completeness properties), so the same proof works over $L$.  %, not any completeness properties% induced from $\mathbb{R}$ (and not any special completeness properties for instance). 
% not any other properties.  
 %but the proof works for any subfield $L$ of $\mathbb{R}$ in exactly the same way since the Fourier-Motzkin method introduced in \cite[Section 1.2]{Ziegler1995} works  %(which is an algorithm%proof works for any subfield $K$ of $\mathbb{R}$,  it is based on a form of matrix reduction (see \cite[Section 1.2]{Ziegler1995}).  
%In our case, this says that the cone defined by the matrix $A$ with integer values, % At any rate, 

We now apply this theorem with $L=\mathbb{Q}$ to the polyhedral cone $P'$ defined by $A\mathbf{x}\leq{0}$. Using the fact that $P'$ is one-dimensional, we see that $P'$ is equal to the cone generated by a single vector $\mathbf{v}$ in $\mathbb{Q}^{n}$. There is then a unique primitive  % with values %the cone defined by $A\mathbf{x}\leq{0}$ (with $A$ an integer-valued matrix) 
%contains a vector $\mathbf{v}$ with rational values and this implies (using the assumption that $P'$ is one-dimensional) that $P'$ contains a unique primitive 
$\mathbb{Z}$-valued vector $\mathbf{w}\in{P'}$ and we can use this vector $\mathbf{w}$ %Here we used the assumption that $P'$ is one-dimensional.   %with integer values. 
%In other words,  we can associate to any edge $e$ in $P$ a primitive $\mathbb{Z}$-valued directional vector $\mathbf{w}$. %for any edge $e$ in the polyhedral cone $P$, we can find a unique primitive $\mathbb{Z}$-valued directional vector $\mathbf{w}$ for $e$.  %that has integer values. 
%this theorem says that for any edge $e$ in the polyhedral cone $P$ with directional vector $\mathbf{v}$, we can find a vector in the span of $\mathbf{v}$ with values in $\mathbb{Q}$. 
%We now use this vector
 to define the {\it{lattice length}} of $e$.

\begin{mydef}\label{LatticeLength}
{\bf{[Lattice length for tropical curves]}} %Length of edges in tropicalizations of curves}}]
Let $e=PQ$ be a finite edge in the tropicalization $\mathrm{trop}(X)$ of a curve $X\subset{(K^{*})^{n}}$ with directional vector $\mathbf{v}$. Let $\mathbf{w}$ be the unique primitive $\mathbb{Z}$-valued vector in the direction of $\mathbf{v}$ as above. %which exists by the considerations above.  % There exists a nontrivial $\mathbb{Q}$-rational vector in the span of $\mathbf{v}$ and thus a primitive $\mathbb{Z}$-valued vector $\mathbf{w}$  in the span of $\mathbf{v}$. 
Writing $\mathbf{v}=\lambda\cdot{\mathbf{w}}$ for some $\lambda\in\mathbb{R}_{\geq{0}}$, we then define $\ell(e)=\lambda$. This is referred to as the {\bf{lattice length}} of $e$. 
\end{mydef}

\begin{exa}\label{MainExample}{\bf{[Main Example]}}
Consider the polynomial $f=x^2y+xy+xy^2+\varpi^{k}$ with $v(\varpi)=1$ and $k$ some positive integer. %The zero set $V(f)$ of $f$ in $(K^{*})^{2}$ defines a curve of genus one, since its homogenization in $\mathbb{P}^{2}$ defines a smooth cubic curve. 
\begin{figure}[h!]
\centering
\includegraphics[width=7cm]{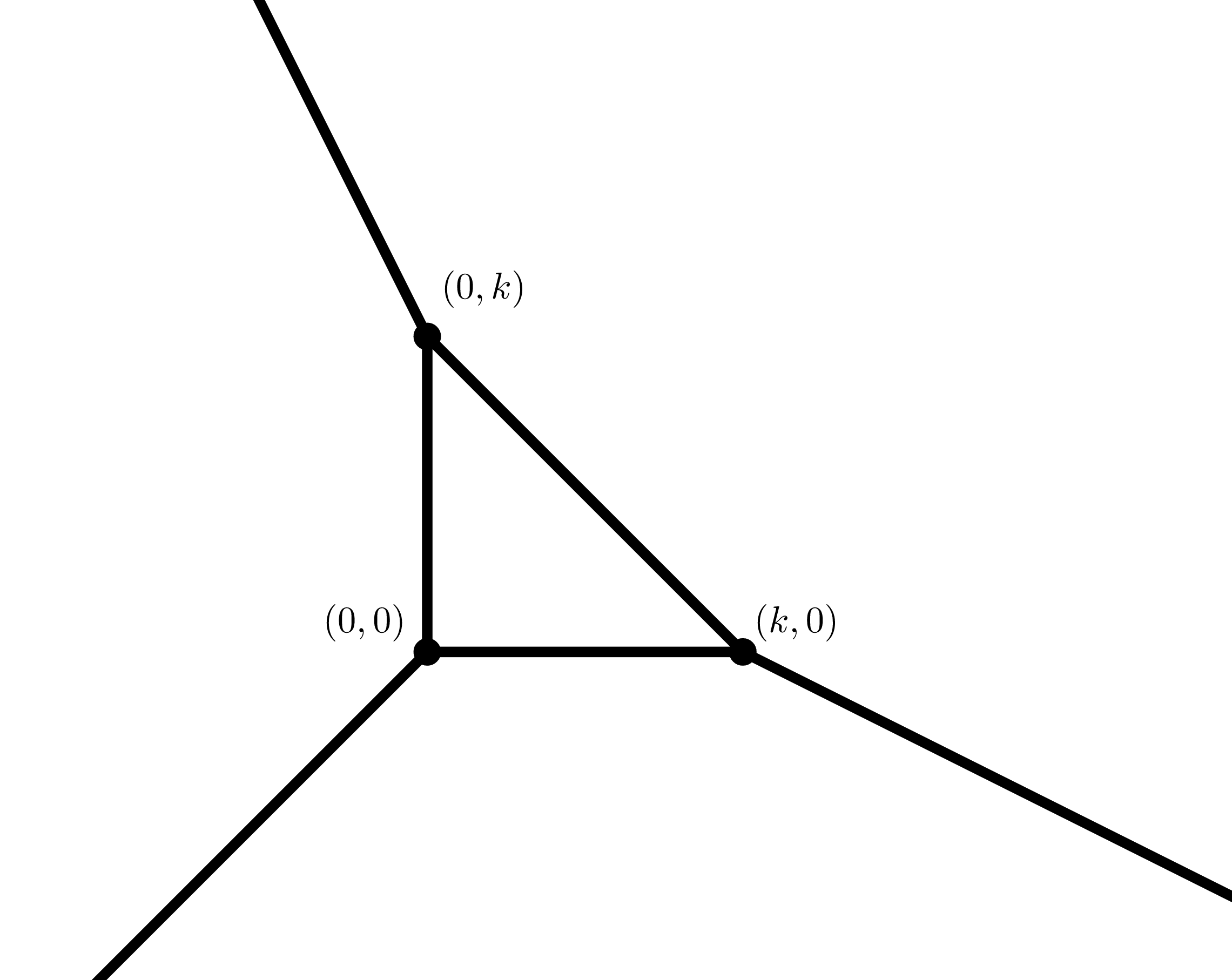}
\caption{\label{Plaatje4} The tropicalizations of the curves in Example \ref{MainExample}.}%
\end{figure}
The tropical polynomial associated to $f$ is   %as %in Example [...]% \ref{NewtonSubdivision1}. % with tropical polynomial 
\begin{equation}
\mathrm{trop}(f)=\mathrm{min}\{2x+y,x+y,x+2y,k\}.
\end{equation}
The tropicalization $\mathrm{trop}(V(f))$ then contains a triangle with vertices $(0,0)$, $(0,k)$ and $(k,0)$, see Figure \ref{Plaatje4}. The {\it{ordinary length}} of the edge between $(0,k)$ and $(k,0)$ is $\sqrt{2}\cdot{k}$. The {\it{lattice length}} of this edge is just $k$, since the primitive directional vector in the direction of $e$ is $\pm(1,-1)$. %corresponding edge in the Newton complex is the edge between $(0,0)$ and $(1,1)$. It is then easy to see that the other edges also have lattice length $k$. 
\end{exa}

%Since the isomorphism class over an algebraically closed field is independent of the rational point $P$, we will also talk about the $j$-invariant associated to an irreducible curve $C\subset{(K^{*})^{n}}$ of genus one.  %chosen, %We will sometimes refer to an irreducible curve $C$ with a rational point $P$ as an elliptic curve, even though it is not projective.

%We now define the notion of a numerically faithful tropicalization for an elliptic curve. To that end, 
We now define the notion of a cycle. 
\begin{mydef}\label{Cycle}
A {\it{cycle}} inside the tropicalization $\mathrm{trop}(C)$ of an irreducible curve $C\subset{(K^{*})^{n}}$ is a finite set of bounded edges inside $\mathrm{trop}(C)$ that form a leafless, connected subgraph of Betti number one. Here a leaf is an edge connected to a vertex of valence one. The {\it{length}} of the cycle is the sum of the lattice lengths (as defined in Definition \ref{LatticeLength}) of the edges in the subgraph.
\end{mydef} 
%\begin{rem}
%As the definition%The reader might think that a curve of genus one%A curve of genus one will never contain
%\end{rem}
\begin{exa}\label{MainExample2}
Let $C$ be the curve defined by $f=x^2y+xy+xy^2+\varpi^{k}=0$, as in Example \ref{MainExample}, which defines an elliptic curve $\overline{C}$. In Section \ref{FinalSection}, we will see a Weierstrass form of this equation.   %Since it defines a nonsingular cubic inside $\mathbb{P}^{2}$ (by homogenizing), its genus is one. 
By the calculations in Example \ref{MainExample}, the corresponding tropicalization contains a triangle with edges of length $k$. This thus defines a cycle inside $\mathrm{trop}(C)$ of total length $3k$. %of length $3k$. 
\end{exa}

\begin{exa}\label{WeierstrassTropExample}

Consider an elliptic curve given by a Weierstrass equation 
 %$(x,y)\mapsto{[x:y:1]}$ and let $E$ be an elliptic curve with $v(j)<0$ defined by 
\begin{equation}
y^2=x^3+Ax+B. 
\end{equation}
\begin{figure}[h!]
\centering
\includegraphics[width=9cm]{{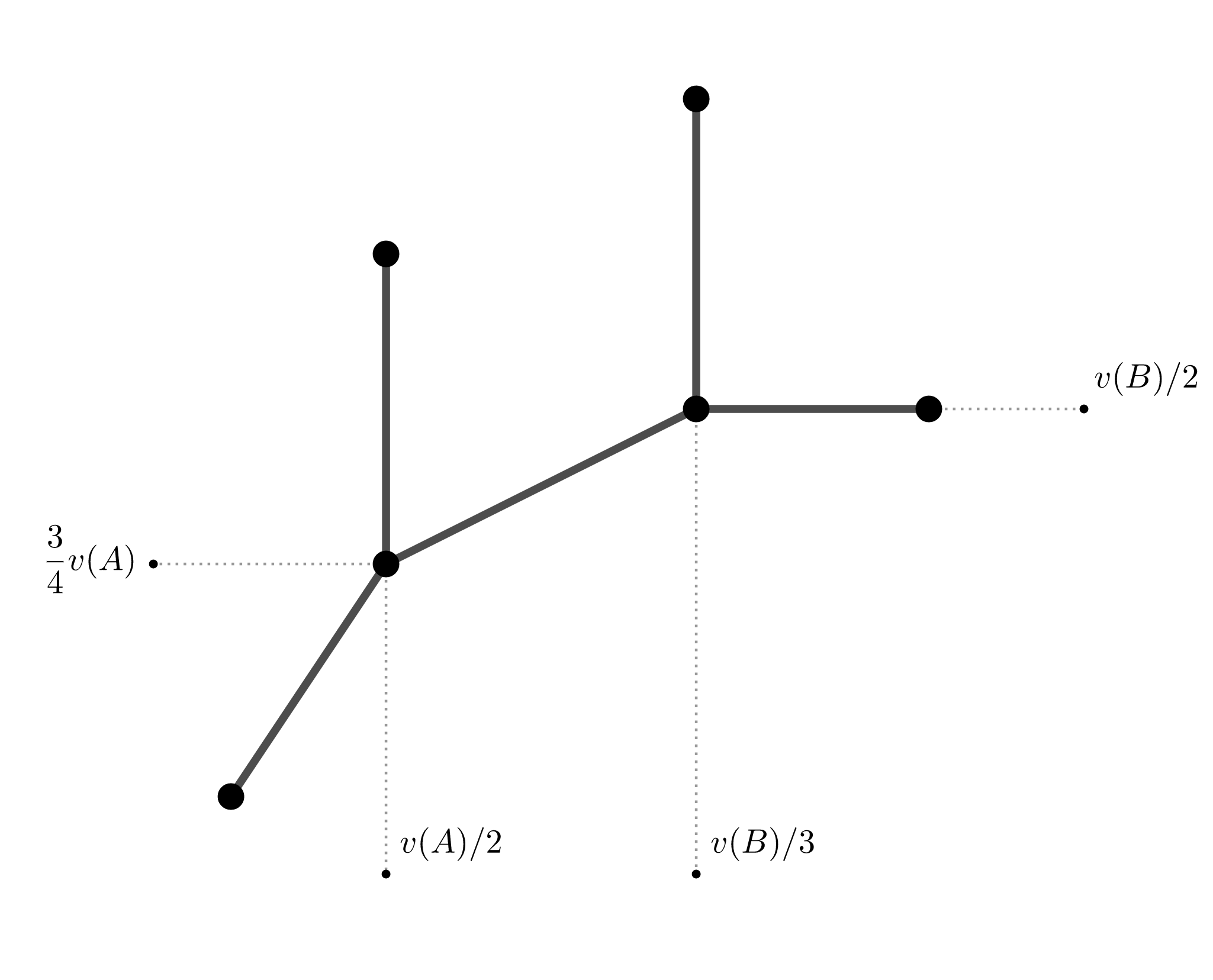}}
\caption{\label{Plaatje547} The tropicalizations in Example \ref{WeierstrassTropExample} of elliptic curves in Weierstrass form with $v(A)/2<v(B)/3$. The dotted lines highlight the $x$ and $y$-coordinates. }%
\end{figure}

We denote the corresponding curve in $(K^{*})^{2}$ by $C$. There are then two options for the tropicalization of $C\subset{(K^{*})^{2}}$.   For $v(A)/2<v(B)/3$, we have that $\mathrm{trop}(C)$ consists of five line segments as in Figure \ref{Plaatje547}. Explicitly, they are given (from left to right) by: 
\begin{itemize}
\item The line $S_{1}$ given by:  $2y=3x$ for $x\leq{}v(A)/2$, $y\leq{3/4\cdot{}v(A)}$,
\item The line $S_{2}$ given by: $x=v(A)/2$ for $y>3/4\cdot{}v(A)$,
\item The line $S_{3}$ given by: $2y=v(A)+x$ for $v(A)/2<x<v(B)/3$,
\item The line $S_{4}$ given by: $x=v(B)/3$ for $y>v(B)/2$,
\item The line $S_{5}$ given by: $y=v(B)/2$ for $x>v(B)/3$. 
\end{itemize}

For $v(A)/2\geq{v(B)/3}$, the two vertical line segments $S_{2}$ and $S_{4}$ become one segment and $S_{3}$ reduces to a single point, as the reader can easily check. In this case, $\mathrm{trop}(C)$ consists of three line segments as in Figure \ref{Plaatje2}. In particular, we now see that the tropicalization of an elliptic curve in Weierstrass form never gives rise to a cycle. %For $A$ and $B$ of positive valuation,
%Then the corresponding tropicalization never contains a cycle, so this tropicalization is never faithful. 
\end{exa}

We now specialize to elliptic curves. Recall from \cite[Chapter III, Proposition 1.4(b)]{Silv1} that to any elliptic curve $E/K$, we can associate  a number $j(E)\in{K}$ such that $j(E)=j(E')$ for two elliptic curves $E,E'$ if and only if $E$ and $E'$ are isomorphic. This number is known as the {\it{$j$-invariant}} of the elliptic curve. This $j$-invariant thus determines the isomorphism class of an elliptic curve over an algebraically closed field. If $K$ is not algebraically closed, then the statement becomes a little bit different, see \cite[Chapter X, Proposition 5.4]{Silv1}. If an elliptic curve $E$ is given in Weierstrass form $y^2=x^3+Ax+B$, then the $j$-invariant can be given explicitly by:  %an explicit formula for the $j$-invariant is given by
\begin{equation}
j(E)=-1728\dfrac{(4A)^3}{\Delta},
\end{equation}
where $\Delta=-16(4A^3+27B^2)$ is the {\it{discriminant}} of the elliptic curve. 

Suppose now that we are given an irreducible curve $C\subset{(K^{*})^{n}}$. %with a rational point %$P$ % of genus one with a rational point 
%$P\in{C(K)}$ on it (since $K$ is algebraically closed, there is no problem in finding such a point). 
Then $C$ can be compactified to give a smooth curve $\overline{C}$ over $K$, see \cite[Chapter 1, Section 6]{Hart1}. If $\overline{C}$ has genus one, then we say that $C$ has genus one. Choosing a rational point on $\overline{C}$ (since $K$ is algebraically closed, there is no problem in finding such a point), we find that $\overline{C}$ is an elliptic curve and thus has a $j$-invariant $j(\overline{C})$. We can then associate a $j$-invariant $j(C)$ to such a curve $C$ by setting $j(C)=j(\overline{C})$. 

We now come to the key definition of a numerically faithful tropicalization. 

\begin{mydef}
{\bf{[Numerically faithful tropicalizations]}}
\label{NumericallyFaithfulDefinition}
Let $C\subset{(K^{*})^{n}}$ be an irreducible curve of genus one and let $\mathrm{trop}(C)$ be its tropicalization, as in Definition \ref{DefinitionTropicalHypersurface}. Suppose that $v(j(C))<0$. We say that $\mathrm{trop}(C)$ is a {\it{numerically faithful tropicalization}} if there is a cycle of of length $-v(j(C))$ in $\mathrm{trop}(C)$, where $j(C)$ is the $j$-invariant associated to $C$. Here the length is as defined in Definition \ref{LatticeLength}. 
\end{mydef}

\begin{exa}
Let $C$ again be an elliptic curve in Weierstrass form, as in Example \ref{WeierstrassTropExample}. We saw in that Example that the tropicalization never contains a cycle. We thus see that this tropicalization is never numerically faithful. 
\end{exa}

\begin{exa}
Let $f$ again be as in Examples \ref{MainExample} and \ref{MainExample2}. We will give an explicit formula for the $j$-invariant of the corresponding elliptic curve in our proof of Theorem \ref{MainTheorem}, see Equation \ref{JinvariantEquation}. A calculation then shows that $-v(j(C))=3k$, showing that it is a numerically faithful tropicalization.  % that the corresponding elliptic curve has $j$-invariant. 
\end{exa}

%We will prove in Theorem \ref{MainTheorem} that every elliptic curve $E$ with $v(j(E))<0$ can be put in the form studied in Example \ref{MainExample}, see Lemma \ref{}. By the previous example, we thus see that every elliptic curve admits a numerically faithful tropicalization. This is the strategy for the upcoming  proof of Theorem \ref{MainTheorem}. %  That is, there exists a curve $C\subset{(K^{*})^{n}}$ with $j(C)=j(E)$ and a cycle in the tropicalization $\mathrm{trop}(C)$ of length $-v(j(C))$. %As mentioned earlier, This result was also obtained in \cite{katz_markwig_markwig_2009}.  %[Markwig/Katz].  %by calculating all possible tropicalizations

%This result is related to the following theorem, which is known as the "{\it{semistable reduction theorem}}" and appears in \cite[Chapter 7, Propositions 5.4 and 5.5]{Silv1} for elliptic curves. If $E$ is an elliptic curve over a discretely valued field $K$ with $v(j(E))<0$, then after a finite extension, the curve $E$ will have (split) semistable reduction. We will prove a version of this theorem in Section \ref{ModelsSection}. %
%\vspace{0.3cm}
%{\flushleft{{\it{Faithful tropicalizations}}
%}}
%\vspace{0.3cm}

\subsection{Faithful tropicalizations}\label{FaithfulTropSection}

We now come to the notion of a {\it{faithful tropicalization}}, as defined in \cite[Section 5.15]{BPR11}. To fully define a faithful tropicalization would take us too far into the field of Berkovich spaces and nonarchimedean geometry, so we will first give a summary of the necessary concepts together with references where the reader can learn more about them. After this, we will state the definition and work out some of the relevant notions in an example. %our examples \ref{MainExample} and \ref{MainExample2}. 
We will then point out some key differences between a {\it{faithful}} tropicalization and a {\it{numerically faithful}} tropicalization. The upcoming material is not strictly necessary to understand the final result of this paper and is somewhat more technical, so the reader can skim through it on a first read-through. We will assume that $K$ is complete throughout this section. 

Let us start with a summary of some of the concepts necessary to define a faithful tropicalization:

\begin{itemize}
\item %Let $X$ be an algebraic variety over $K$. In ...%One can associate an analytic   %
To any connected curve $X$ over a nonarchimedean field $K$, one can associate its so-called {\it{Berkovich analytification}} $X^{\mathrm{an}}$, as in \cite{berkovich2012}.  %of $X$. %This is also known as the Berkovich space associated to $X$. 
It is a path-connected Hausdorff space that contains the rational points $X(K)$ as a dense open subset and locally looks like a {\it{metric tree}}.  %It can be thought of as the intrinsic object that any faithful tropicalization is ca, in the sense that it gives an {\it{isometry}} on a skeleton $\Sigma$ of $X^{\mathrm{an}}$. 
Some good references to learn more about Berkovich spaces are:  \cite{BakerRumely}, \cite{Jonsson1} and \cite{berkovich2012}.
\item Let $M\simeq{\mathbb{Z}^{n}}$ be a lattice, $N=\mathrm{Hom}(M,\mathbb{Z})$ its dual lattice and ${T}=\mathrm{Spec}(K[M])$ the corresponding torus, written as the spectrum of the group ring over $K$ corresponding to $M$. We write $N_{\mathbb{R}}:=N\otimes{\mathbb{R}}\simeq{\mathbb{R}^{n}}$ for the real vector space corresponding to $N$. 
Consider a toric variety $Y_{\Delta}$, defined by a polyhedral fan $\Delta$ in $N_{\mathbb{R}}$. For every polyhedral cone $\sigma$ in $\Delta$ there is a tropicalization map from $Y_{\sigma}^{\mathrm{an}}$ to the space of additive semigroup homomorphisms $\mathrm{Hom}(\sigma^{\vee}\cap{M},\mathbb{R}\cup{\{\infty\}})$ and we write $N_{\mathbb{R}}(\sigma)$ for the image of $Y_{\sigma}^{\mathrm{an}}$ under this map.  For an inclusion of cones $\tau\preceq{\sigma}$, this yields $N_{\mathbb{R}}(\tau)\subset{N_{\mathbb{R}}(\sigma)}$ and we write $N_{\mathbb{R}}(\Delta)$ for the space obtained by gluing along these inclusions induced by cones in $\Delta$ (see \cite[Definition 4.3]{BPR11} for the details). This then induces a % is then a 
{\it{generalized tropicalization map:}}
\begin{equation}\label{GeneralizedTrop}
\mathrm{trop}:Y_{\Delta}^{\mathrm{an}}\rightarrow{N_{\mathbb{R}}(\Delta)}.
\end{equation} 

In terms of the tropicalizations studied in the beginning of Section \ref{Tropicalizations}, these concepts translate as follows. %This generalized tropicalization map corresponds to  %The corresponding tropicalizations studied in the beginning of Section \ref{Tropicalizations} are as follows.  %and Equation \ref{NaiveMap} as follows. %In the context of the tropicalizations studied in the beginning of Section \ref{Tropicalizations} and Equation \ref{NaiveMap}, the corresponding objects are as follows. 
The toric variety used there is  %In the scenario of Equation \ref{NaiveMap}, the toric variety 
$(K^{*})^{n}$, given in scheme-theoretic form by $Y_{\Delta}=\mathrm{Spec}(K[\mathbf{x},\mathbf{x}^{-1}])$, and $N_{\mathbb{R}}(\Delta)$ is just $\mathbb{R}^{n}$. The $K$-rational points of $Y_{\Delta}$ (which can be identified with $(K^{*})^{n}$) naturally embed into $Y^{\mathrm{an}}_{\Delta}$ and the generalized tropicalization map in Equation \ref{GeneralizedTrop} coincides on these $K$-rational points with the tropicalization map $\mathrm{val}(\cdot{})$ given in Equation \ref{NaiveMap}. For more background information regarding toric varieties and generalized tropicalization maps, we refer the reader to \cite[Chapter 6]{tropicalbook} and \cite[Sections 4 and 5]{BPR11}. Other good references for toric varieties are \cite{FultonToric} and \cite{Cox2011}.  %The generalized tropicalization map mentioned above coincides with the tropicalization map defined in Section [] on the type-$1$ points. 

\item One can also use this generalized tropicalization map to define a tropicalization map for closed subschemes of a toric variety $Y_{\Delta}$. In the context of curves and their compactifications, this tropicalization map works as follows. Let $X$ be a smooth and connected curve over $K$ and let $\overline{X}$ be its smooth compactification as in \cite[Chapter 1, Section 6]{Hart1}. Let $\overline{X}\rightarrow{Y_{\Delta}}$ be a closed immersion of $\overline{X}$ into a toric variety $Y_{\Delta}$ with polyhedral fan $\Delta$ such that $\overline{X}$ meets the dense torus $T$ of $Y_{\Delta}$. % such that $X\subset{{T}}$, where ${T}$ is the dense torus in $Y_{\Delta}$. 
This then induces a %generalized tropicalization map in Equation \ref{GeneralizedTrop} then induces a 
tropicalization map  %inside $N_{\mathbb{R}}\simeq{\mathbb{R}^{n}}$, one can construct a {\it{generalized tropicalization map}}
\begin{equation}\label{GlobalTropicalization}
\mathrm{trop}:\overline{X}^{\mathrm{an}}\rightarrow{N_{\mathbb{R}}(\Delta)}
\end{equation}
by taking the composition %$i\circ{\mathrm{trop}}%composing 
of the embedding and the generalized tropicalization map in Equation \ref{GeneralizedTrop}, %in Equation \ref{GeneralizedTrop}, 
%by first embedding $\overline{X}^{\mathrm{an}}$ into $Y_{\Delta}$ and  under the generalized tropicalization map, 
see \cite[Definition 4.3]{BPR11}. %In general, this idea can be used to construct a generalized tropicalization map on a closed subscheme of a toric variety. 
We will use this map to define faithful tropicalizations of the proper curve $\overline{X}$. %One can also use this idea to define the tropicalization map from a 

%In other words, using the language of proper toric varieties we can now tropicalize proper varieties such as $\overline{X}$. 
%If $X\xhookrightarrow{\mathbb{T}}$ in the dense torus $\mathbb{T}$ of $Y_{\Delta}$

%In this paper, the toric variety is just $(K^{*})^{n}$ and the tropicalization map defined in Equation \ref{NaiveMap} coincides with the generalized tropicalization map on the $K$-rational points $X(K)\subset{X^{\mathrm{an}}}$. 
 %Every subgraph $\Gamma\subset{X^{\mathrm{an}}}$ with bounded edges inherits a natural metric, the {\it{skeletal metric}}, giving it the structure of a metric graph. The image of $\mathrm{trop}(\cdot{})$ is then an abstract one-dimensional $\Gamma$-rational polyhedral complex and the metric on this image is given by the lattice length, see \cite[...]{BPR11}. We invite the reader to compare this with [...]. 

%See \cite[Section 5.58]{BPR11} for more on this metric and its implications. %, turning it into a metric graph, see \cite[Section 5.58]{BPR11}. 
\item Let $\overline{\mathcal{X}}$ be a semistable $R$-model for $\overline{X}$. That is, it is an integral scheme $\overline{\mathcal{X}}$ with a flat proper morphism $\overline{\mathcal{X}}\rightarrow{\mathrm{Spec}(R)}$ and an isomorphism $\overline{\mathcal{X}}_{\eta}\simeq{\overline{X}}$ such that the special fiber $\overline{\mathcal{X}}_{s}$ is a semistable curve in the sense of \cite[Chapter 10, Definition 3.1]{liu2}. %reduced and contains only ordinary double points as its singularities (see \cite[Definition 7.5.13]{liu2} for a definition of this last concept).

%This consists of three requirements: 1) $\overline{\mathcal{X}}$ is an integral scheme and $\overline{\mathcal{X}}\rightarrow{\mathrm{Spec}(R)}$ is a flat, proper morphism, 2) there is an isomorphism %with $\overline{\mathcal{X}}$ an integral scheme, an isomorphism 
%$\overline{\mathcal{X}}_{\eta}\simeq{\overline{X}}$, and 3) %the requirement %$\overline{\mathcal{X}}$ to the spectrum of the valuation ring $R$ such
%the special fiber $\overline{\mathcal{X}}_{s}$ has to be reduced and it has to contain only ordinary double points as singularities  %, together with an isomorphism .   %These ordinary double points are characterized by  
To any semistable $R$-model $\overline{\mathcal{X}}$ of $\overline{X}$, one can associate a {\it{skeleton}} as follows\footnote{In \cite{BPRa1} they use formal semistable $R$-models to define the skeleton, but \cite[Lemma 5.1]{ABBR1} tells us that the category of semistable models is equivalent to the category of formal semistable $R$-models by using the "completion functor", so there is no harm in using this category.}. The generic points of the special fiber $\overline{\mathcal{X}}_{s}$ give rise to type-$2$ points of the Berkovich analytification as in Theorem \cite[Theorem 4.6(1)]{BPRa1} using the reduction map. This set of type-$2$ points is known as a semistable vertex set and we denote it by $V$. The complement $X^{\mathrm{an}}\backslash{V}$ then decomposes as a disjoint union of open balls and generalized open annuli. We define the skeleton $\Sigma(\overline{X},V)$ to be the union of the skeleta of these generalized open annuli together with the set $V$, see \cite[Definition 3.3]{BPRa1}. 
\item There is a metric on the set $\mathbf{H}_{0}(\overline{X}^{\mathrm{an}})$ of skeletal points, which is the set of points of type $2$ and $3$ inside $\overline{X}^{\mathrm{an}}$. This metric is known as the {\it{skeletal metric}}. %:=\overline{X}^{\mathrm{an}}\backslash{\overline{X}(K)}$, known as the {\it{skeletal metric}}. 
Explicit local formulas for this metric are given in \cite[Section 5.3]{BPRa1}. A {\it{finite subgraph}} of $\overline{X}^{\mathrm{an}}$ is then an isometric embedding $\Gamma\rightarrow\mathbf{H}_{0}(\overline{X}^{\mathrm{an}})$ of a finite connected metric graph $\Gamma$ into this set of skeletal points. %, where explicit local formulas are also given. %. Any skeleton   

%It is a finite subgraph of $X^{\mathrm{an}}$ and by the previous comment it inherits a natural metric, turning it into a metric graph. 

%Let %...Skeleton/Semistable models... %For a nonzero function $f$ in the function field $K(C)$ of $C$, the logarithm $-\mathrm{log}|f|$ restricts to a piecewise linear function on $\Gamma$.   %This map coincides with the tropicalization map defined in \ref{NaiveMap} on the $K$-rational points $X(K)\subset{X^{\mathrm{an}}}$ if the toric variety is just $(K^{*})^{n}$. %For a closed immersion of 
\end{itemize}

To familiarize ourselves with some of the concepts given here, let us give an example of a skeleton using an explicit semistable model. We will work out the semistable model and the skeleton for the curve defined in Example \ref{MainExample}.  % using semistable models.

\begin{exa}
\label{SkeletonSemistable}
Let $f=x^2y+xy+xy^2+\varpi^{k}$ be the polynomial from Examples \ref{MainExample} and \ref{MainExample2} defining a curve inside $(K^{*})^{2}$. We consider the (affine) model
\begin{equation}
\mathcal{C}=\mathrm{Spec}(R[x,y]/(f))
\end{equation}
over the valuation ring $R$ of $K$. The special fiber is then given by the spectrum of the tensor product $R[x,y]/(f)\otimes_{R}{k}=k[x,y]/(\overline{f})$, where $k$ is the residue field and $\overline{f}$ is the reduction of $f$ to the polynomial ring $k[x,y]$. We then have $\overline{f}=x^2+xy+xy^2$. This polynomial is reducible: $\overline{f}=xy(x+y+1)$. We now see that the special fiber contains three irreducible components given by the factors $x$, $y$ and $x+y+1$. The corresponding prime ideals in $\mathcal{C}$ are given by $\mathfrak{p}_{1}=(x,\varpi)$, $\mathfrak{p}_{2}=(y,\varpi)$ and $\mathfrak{p}_{3}=(x+y+1,\varpi)$. The corresponding components $\Gamma_{i}=V(\mathfrak{p}_{i})$ intersect each other transversally in a cyclic fashion: the component $\Gamma_{1}$ intersects $\Gamma_{2}$ in the maximal ideal $\mathfrak{n}_{1,2}=(x,y,\varpi)$, $\Gamma_{2}$ intersects $\Gamma_{3}$ in $\mathfrak{n}_{2,3}=(y,\varpi,x+1)$ and $\Gamma_{3}$ intersects $\Gamma_{2}$ in $(x,y+1,\varpi)$. The incidence graph (see \cite[Chapter 10, Definition 3.17]{liu2}) %and \cite[Section 4.9]{BPRa1}) of $\mathcal{C}$ 
thus consists of a cyclic graph with three vertices and three edges, which already shows a strong connection with the tropicalization of $C$ calculated in Example \ref{MainExample2}. 

Using some birational geometry, this $\mathcal{C}$ can be considered as an open affine subset of a semistable model $\overline{\mathcal{C}}$ for the compactified $\overline{C}$, where the generic points of the nontrivial components of $\overline{\mathcal{C}}$ are all in this open affine $\mathcal{C}$. By Theorem \cite[Theorem 4.7]{BPRa1}, this gives a so-called {\it{semistable vertex set}} $V$ of $\overline{C}$ and the skeleton $\Sigma(\overline{C},V)$ is just the incidence graph of $\mathcal{C}$ (see \cite[Section 4.9]{BPRa1}). A quick calculation shows that the lengths of the ordinary double points are $k$ (with respect to $\varpi$, see \cite[Chapter 10, Definition 3.23]{liu2} for the definition), which again shows the strong connection with the tropicalization in Example \ref{MainExample2}.  %so the lengths of the edges are also $k$. %This means that we can identify these edges with the intervals $[0,k]$, but glued together at their endpoints.    %(the punctures $D$ are the points at infinity in the homogenization) and thus a skeleton $\Sigma(X,V)$, which is just the incidence graph 
%set of type-$2$ points of the Berkovich skeleton
%In terms of \cite{BPR1}, we now find that 
 %three vertices with three edges%We say that the intersection graph of the special fiber consists  %can thus conclude that the spe
\end{exa}

\begin{mydef}
{\bf{[Faithful tropicalizations]}}
(See \cite[Section 5.15]{BPR11}) Let $X$ be a smooth, connected curve over a complete, algebraically closed field $K$ with smooth compactification $\overline{X}$ and consider a closed immersion $\overline{X}\rightarrow{Y_{\Delta}}$ of $\overline{X}$ into a toric variety $Y_{\Delta}$ with dense torus $T$.  Assume that $T\cap{\overline{X}}\neq\emptyset$ and consider the tropicalization map given in Equation \ref{GlobalTropicalization}. 
\begin{itemize}
\item We say that a finite subgraph $\Gamma$ of the Berkovich analytification $\overline{X}^{\mathrm{an}}$ is {\it{faithfully tropicalized}} by the tropicalization map $\mathrm{trop}:\overline{X}^{\mathrm{an}}\rightarrow{N_{\mathbb{R}}(\Delta)}$ if $\mathrm{trop}(\cdot{})$ maps $\Gamma$ homeomorphically and isometrically onto its image. % in ${N}_{\mathbb{R}}\simeq{}\mathbb{R}^{n}$.
\item We say that $\mathrm{trop}(\cdot{})$ is a faithful tropicalization of $\overline{X}$ if a skeleton of $\overline{X}$ is faithfully tropicalized. 
\end{itemize}
\end{mydef}

\begin{exa}\label{FaithfulTrop3}
It can be shown that any polynomial $f\in{K[x,y]}$ whose {\it{Newton complex}} (see \cite{Rabinoff2012} and \cite[Proposition 3.1.6]{tropicalbook}) is a unimodular triangulation gives rise to a faithful tropicalization. This is \cite[Corollary 5.28(2)]{BPR11}.
\begin{figure}[h!]
\centering
\includegraphics[width=6cm]{{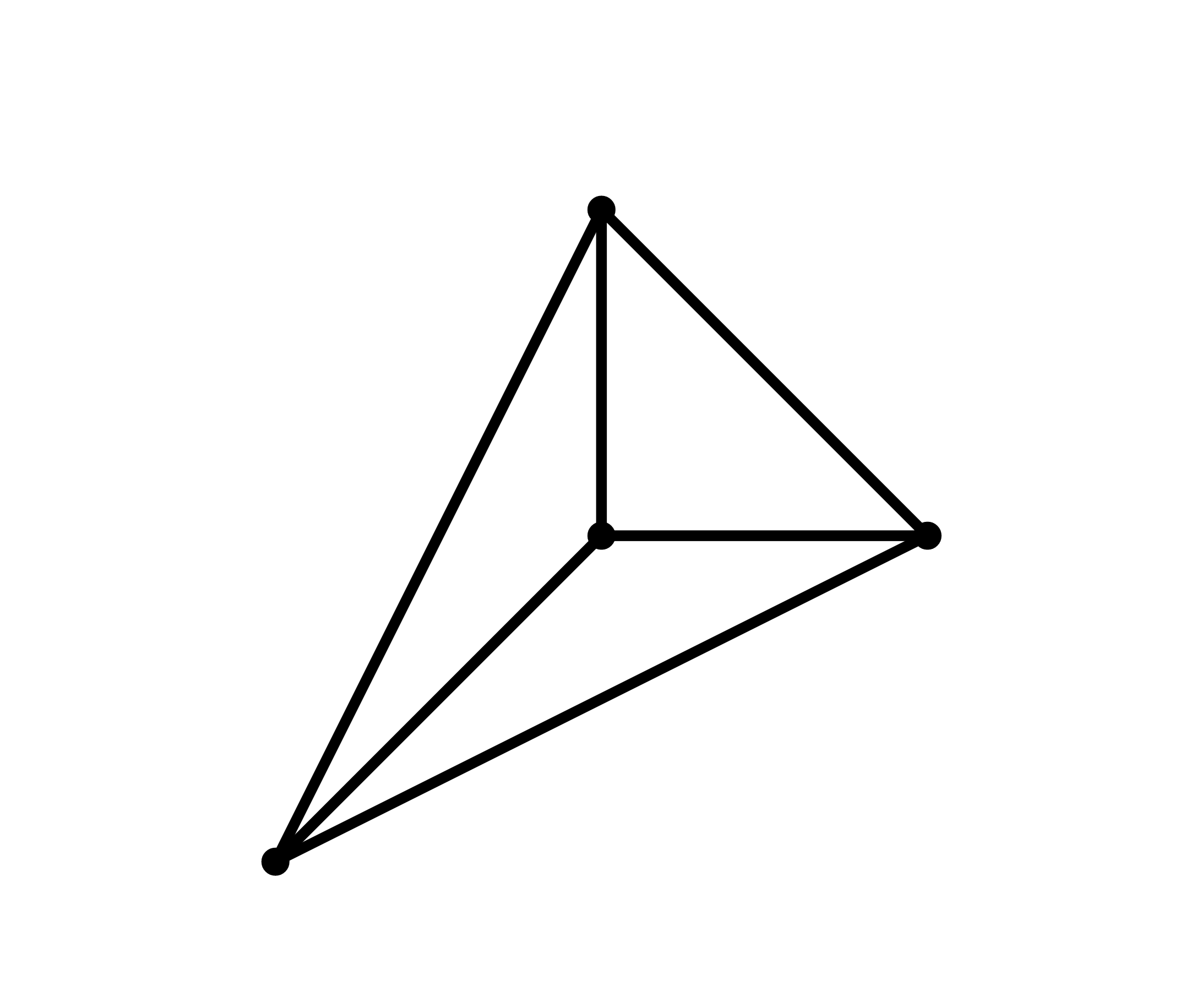}}
\caption{\label{Plaatje548} The Newton complex in Example \ref{FaithfulTrop3}.} %of elliptic curves in Weierstrass form with $v(A)/2<v(B)/3$. The dotted lines highlight the $x$ and $y$-coordinates. }%
\end{figure}
 A quick calculation shows that the Newton complex defined by $f=x^2y+xy+xy^2+\varpi^{k}$ is as in Figure \ref{Plaatje548}, which is a unimodular triangulation. This implies that the corresponding tropicalization $\mathrm{trop}(C)$ is faithful and we thus conclude that $v(j(C))=-3k$ from Example \ref{MainExample2}. In Theorem \ref{MainTheorem}, we will show that $v(j(C))=-3k$ without this result on faithful tropicalizations. %, since a faithful tropicalization is automatically numerically faithful. 
\end{exa}

We now compare the notions of a numerically faithful tropicalization and a faithful tropicalization. %We will assume that the toric variety in the definition of a faithful tropicalization is just $(K^{*})^{n}$ with the corresponding trivial polyhedral fan for $N_{\mathbb{R}}$.  
First, a faithful tropicalization is automatically numerically faithful (if we adopt the more general form of tropicalizations using toric varieties as explained earlier). Indeed, if $v(j(E))<0$, then the Berkovich space contains a cycle with length $-v(j(E))$ and this subgraph is mapped isometrically (for the lattice length on $\mathrm{trop}(C)$) onto its image in $\mathrm{trop}(C)$, so $\mathrm{trop}(C)$ contains a cycle of length $-v(j(E))$. See the discussion in \cite[Section 6]{BPR11} for more details.  It is not true however that every numerically faithful tropicalization is faithful, as the following example shows. %We changed the example that was present in the second arXiv version of \cite{BPR11}  

%The converse is not true however, as we will now demonstrate. % and %we will now demonstrate a counter-example. 
\begin{exa}\label{NonFaithfulExample}
We will use the slope formula from \cite[Theorem 5.15]{BPRa1} to calculate the abstract tropicalization as in \cite[Theorem 6.2]{BPR11}. Throughout this example we will be making heavy use of the results in \cite{BPR11} and \cite{BPRa1}, so we refer the reader to those papers for more information on these concepts. 
\begin{figure}[h!]
\centering
\includegraphics[width=8cm]{{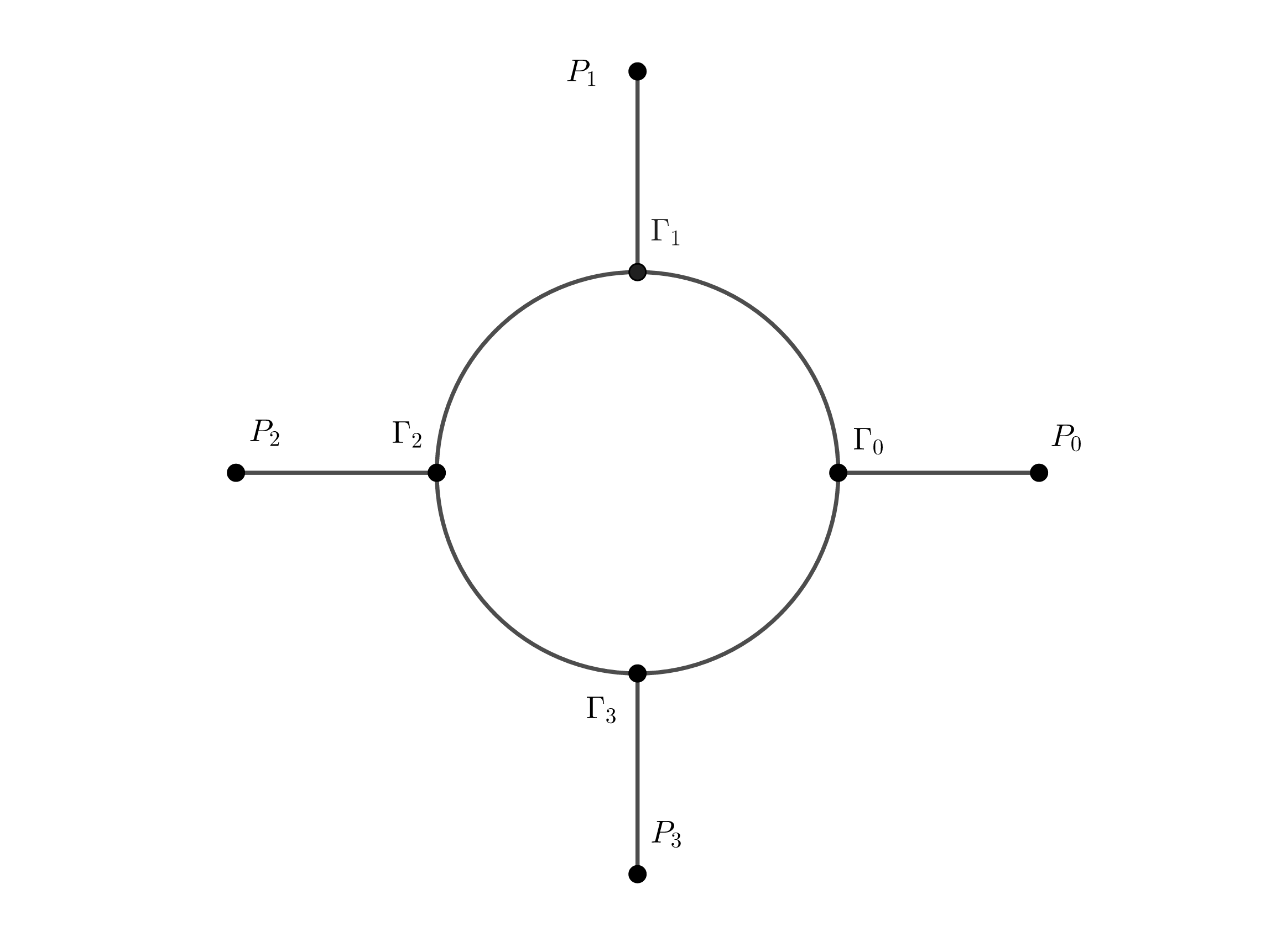}}
\caption{\label{Plaatje550} The Berkovich minimal skeleton in Example \ref{NonFaithfulExample}, together with a set of four-torsion points that retract nontrivially onto the skeleton.  }%
\end{figure}
Consider an elliptic curve $E$ with multiplicative reduction, that is: $v(j(E))<0$. Then $E$ can be analytically uniformized:
\begin{equation}\label{UniformizationEquation}
E^{\mathrm{an}}\simeq{\mathbb{G}^{\mathrm{an}}_{m}/q^{\mathbb{Z}}}
\end{equation}
for some $q\in{K^{*}}$ with $v(q)=-v(j(E))$. This is also known as the Tate uniformization.  

%This yields the following sequence of maps
%\begin{equation}
%K^{*}\rightarrow{\ov}
%\end{equation} 
We now choose a fourth root $q^{1/4}$ of $q$ and we consider the four-torsion point $P\in{E(K)[4]}$ corresponding to it under the Tate uniformization above. %a type-$1$ point $q^{1/4}$ under the above isomorphism. 
We write $P_{i}=i\cdot{P}$ for the multiples of $P$ under the group law on $E(K)$. In particular we have that $P_{0}=P_{4}$ is the identity element. The minimal skeleton $\Sigma$ of $E$ is then a circle of length $-v(j(E))$ and there is a natural retraction map $E^{\mathrm{an}}\rightarrow{\Sigma}$ onto this skeleton. If we consider the natural map $K^{*}\rightarrow{E^{\mathrm{an}}}$ arising from the Tate uniformization, then the composite map $K^{*}\rightarrow{\Sigma}$ is just given by $z\mapsto{[\mathrm{val}(z)]}$, see the proof of \cite[Theorem 6.2]{BPR11}. % and $E %\cite{%It is easy to see what this map is by composing $K^{*}\rightarrow{\mathbb{G}^{\mathrm{an}}_{m}/q^{\mathbb{Z}}\simeq{E^{\mathrm{an}}}
We denote the type-$2$ points the points $P_{i}$ retract to by $\Gamma_{i}$. See Figure \ref{Plaatje550} for a pictorial description of this retraction.

Now consider the following divisors:
\begin{align*}
D_{1}&=6P_{2}-4P_{1}-2P_{3},\\
D_{2}&=2P_{3}-P_{2}-P_{0}.
\end{align*}
Since they have degree zero and add up to zero under the group law, we find that they are principal by \cite[Chapter III, Corollary 3.5]{Silv1}. We can thus write $\mathrm{div}(f_{i})=D_{i}$ for $f_{i}$ in the function field of $E$. We calculate the corresponding piecewise-linear functions on the skeleton $\Sigma$ using the slope formula. The result is in Figures \ref{Plaatje551} and \ref{Plaatje549}. We thus see that both $F_{i}=-\mathrm{log}{|f_{i}|}$ have slope zero on $\Gamma_{0}\Gamma_{1}$ and slope divisible by $2$ on $\Gamma_{1}\Gamma_{2}$. The slope of $F_{2}$ on $\Gamma_{2}\Gamma_{3}$ and $\Gamma_{3}\Gamma_{0}$ is $\pm{1}$, whereas the slope of $F_{1}$ is $-2$ on these edges. The slope of $F_{1}$ on $\Gamma_{1}\Gamma_{2}$ is $4$. 

\begin{figure}[h!]
\centering
\includegraphics[width=5cm]{{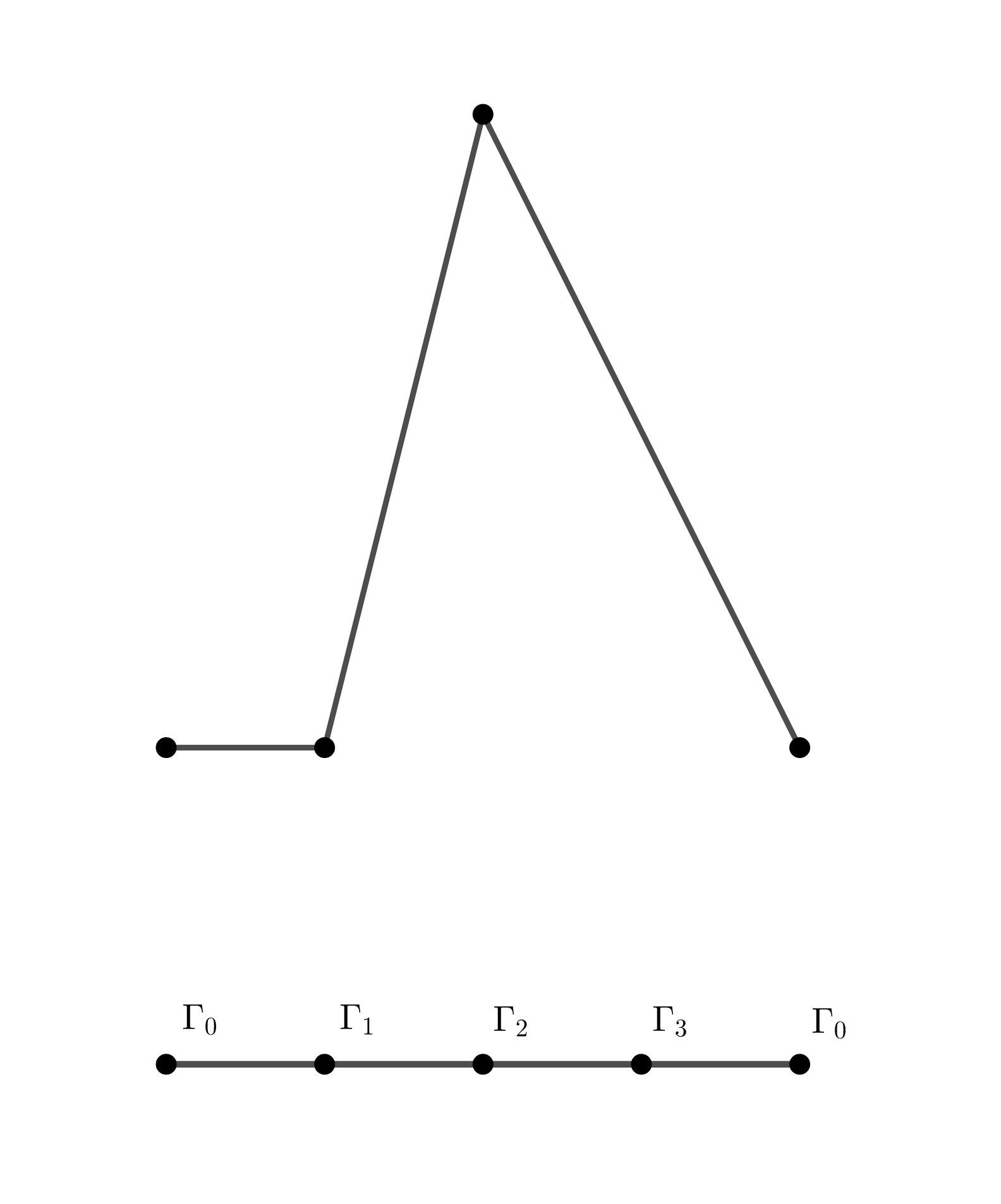}}
\caption{\label{Plaatje551} The piecewise linear function $F_{1}$ constructed in Example \ref{NonFaithfulExample}. }%The tropicalizations in Example \ref{WeierstrassTropExample} of elliptic curves in Weierstrass form with $v(A)/2<v(B)/3$. The dotted lines highlight the $x$ and $y$-coordinates. }%
\end{figure}

\begin{figure}[h!]
\centering
\includegraphics[width=7cm]{{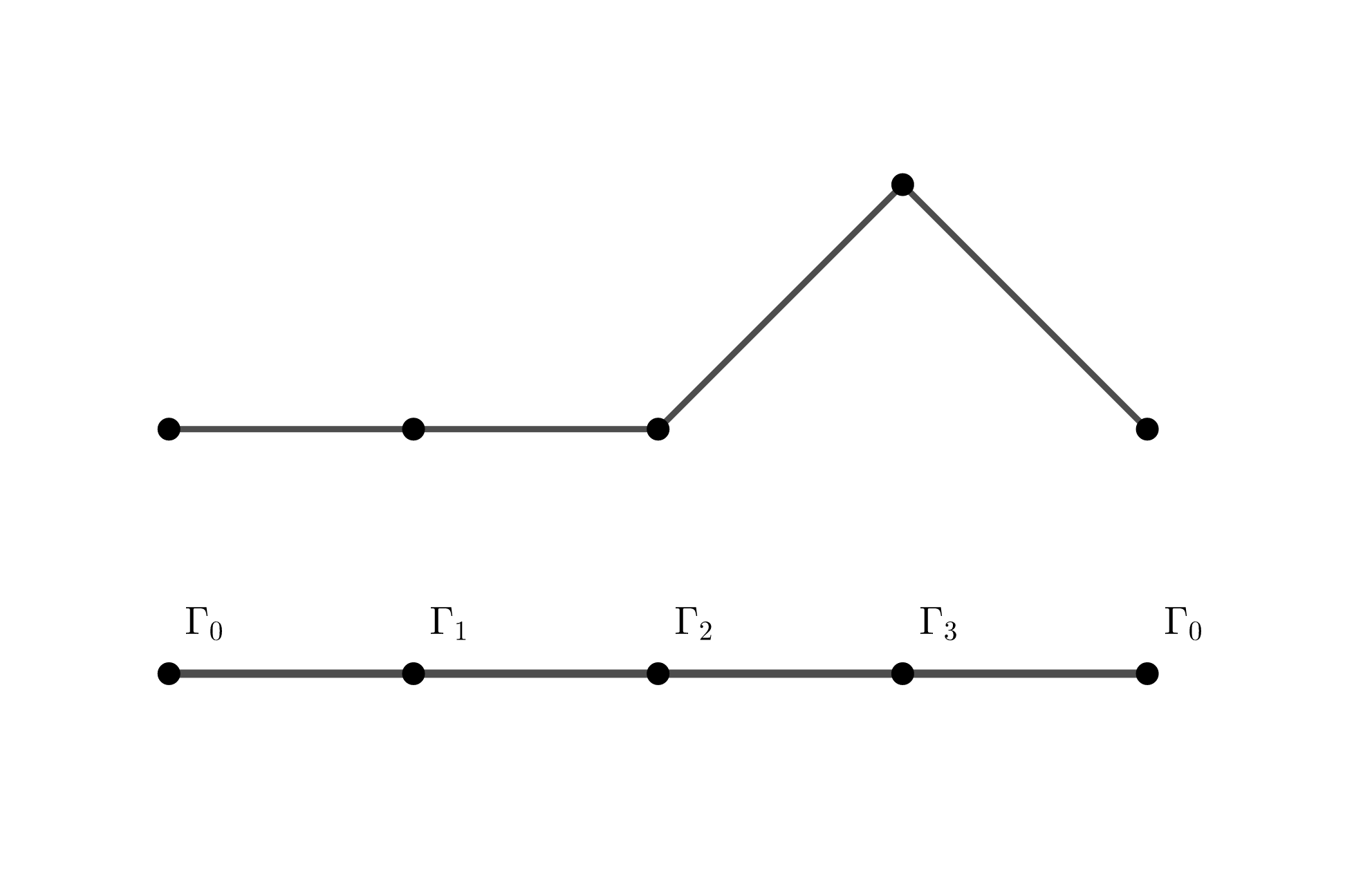}}
\caption{\label{Plaatje549} The piecewise linear function $F_{2}$ constructed in Example \ref{NonFaithfulExample}. }%The tropicalizations in Example \ref{WeierstrassTropExample} of elliptic curves in Weierstrass form with $v(A)/2<v(B)/3$. The dotted lines highlight the $x$ and $y$-coordinates. }%
\end{figure}

If we now consider the embedding given by $(f_{1},f_{2})$, it might not be a closed embedding. There are two ways out of this problem. Consider the union $Z$ of the supports of the $f_{i}$ inside $E$. Let $U=E\backslash{Z}$.  We can then define a generalized tropicalization map on $U$. To that end, we first define a map\footnote{To do this, it is enough to define a ring map $K[x,x^{-1},y,y^{-1}]\rightarrow{\mathcal{O}_{E}(U)}$, see \cite[Chapter 2, Proposition 3.25]{liu2}. This is given by mapping $x\mapsto{f_{1}}$ and $y\mapsto{f_{2}}$. It is well-defined because the $f_{i}$ are invertible on $U$. } $U\rightarrow{(K^{*})^{2}}$ using these $f_{i}$ and we then define an induced tropicalization map from $U^{\mathrm{an}}$ to $\mathbb{R}^{2}$, see \cite[Section 8]{SkeletonJacobian}.  % the isometric embeddings of finite subgraphs of $U^{\mathrm{an}}$ into $N_{\mathbb{R}}\$. % We can then talk about where the map is defined and still talk about isometric embeddings of finite subgraphs of $U^{\mathrm{an}}$ into $N_{\mathbb{R}}=\mathbb{R}^{n}$. 
%For this approach, see \cite[Section 8]{SkeletonJacobian}. 
The tropicalization induced by $(f_{1},f_{2})$ in this sense then has the following properties. It collapses $\Gamma_{0}\Gamma_{1}$ to a point since both $F_{i}$ have slope zero on that segment. It expands $\Gamma_{1}\Gamma_{2}$ by a factor $2$, since the greatest common divisor of the slopes of the $F_{i}$ is $2$, see \cite[Remark 5.6]{BPR11}. On the remaining two edges the greatest common divisor is just one, so it is isometric there. This implies that the tropicalization contains a cycle of length $0+2\cdot{}(v(q)/4)+1\cdot{}(v(q)/4)+1\cdot{}(v(q)/4)=v(q)=-v(j(E))$, so it is numerically faithful. It is not faithful however, since one of the segments of the minimal skeleton is contracted and another is expanded.    %In other words, if we allow  %Locally, this tropicalization map is again given by $P\rightarrow{}%In terms  

We can also modify the embedding to make it a closed embedding. The following idea can also be found in an earlier version of \cite{BPR11}. We first construct a closed embedding with a "trivial tropicalization" on the minimal skeleton. That is, we construct two functions whose piecewise linear functions are trivial on $\Sigma$, but still define a closed embedding. After that, we combine the two to make a closed embedding with the desired tropicalization. 

We can in fact use the argument in \cite[Theorem 6.2]{BPR11}, but with different points\footnote{The old version of \cite{BPR11} uses a mix of a two-torsion point and a three-torsion point. In our argument, we only use three-torsion points. }. Choose a primitive third root of unity $\zeta_{3}\in{K^{*}}$ and consider the point $T$ in $E(K)[3]$ that corresponds to it under the Tate uniformization. We will denote the multiples of $T$ by $T_{i}$ again, where $T_{1}=T$ and $T_{0}=T_{3}$ is the identity. For a pictorial description, see Figure \ref{Plaatje591}. 
\begin{figure}[h!]
\centering
\includegraphics[width=7cm]{{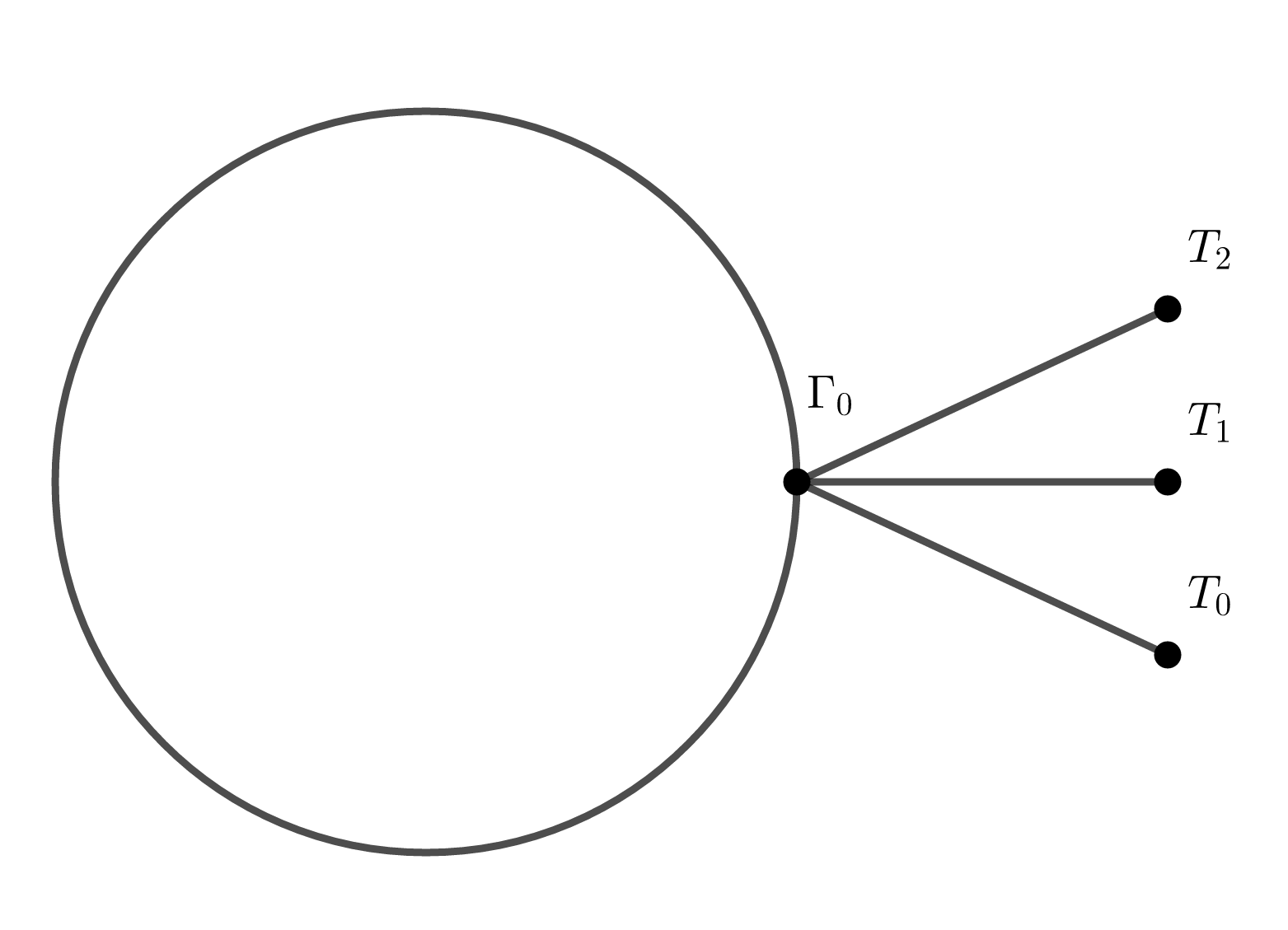}}
\caption{\label{Plaatje591} The minimal skeleton in Example \ref{NonFaithfulExample} with three-torsion points retracting to the same point on $\Sigma$.} %The piecewise linear function $F_{2}$ constructed in Example \ref{NonFaithfulExample}. }%The tropicalizations in Example \ref{WeierstrassTropExample} of elliptic curves in Weierstrass form with $v(A)/2<v(B)/3$. The dotted lines highlight the $x$ and $y$-coordinates. }%
\end{figure}

 The divisors $D_{3}=2T_{1}-T_{2}-T_{0}$ and $D_{4}=2T_{2}-T_{1}-T_{0}$ are principal and we can write $\mathrm{div}(f_{3})=D_{3}$ and $\mathrm{div}(f_{4})=D_{4}$. The functions $f_{3}$ and $f_{4}$ are then  global generating sections of the sheaf $\mathcal{O}_{E}(D)$ (see \cite[Section 7.1.2]{liu2}), where $D=T_{0}+T_{1}+T_{2}$. Furthermore, the corresponding map $E\rightarrow{\mathbb{P}^{2}}$ is a closed embedding since $\mathrm{deg}(D)>2$, see \cite[Chapter 7, Proposition 4.4(b)]{liu2}. Note that the piecewise linear functions of $f_{3}$ and $f_{4}$ are trivial on the skeleton $\Sigma$. Indeed, the $T_{i}$ all retract to the same point on $\Sigma$. We thus see that the corresponding tropicalization does not contain a cycle.

We now have three morphisms: $E\rightarrow{\mathbb{P}^{2}}$ arising from the pair $(f_{3},f_{4})$, and two morphisms $E\rightarrow{\mathbb{P}^{1}}$ arising from $f_{1}$ and $f_{2}$. %another $E\rightarrow{\mathbb{P}^{1}}$ from $f_{2}$. 
Using the universal property of the product, we obtain a morphism $E\rightarrow{\mathbb{P}^{2}\times{\mathbb{P}^{1}}\times{\mathbb{P}^{1}}}$, which is a closed embedding because it is a closed embedding on the first factor. The embedding now tropicalizes as follows. First, note that the tropical functions $-\mathrm{log}|f_{3}|$ and $-\mathrm{log}|f_{4}|$ are constant on $\Sigma$, so their coordinates don't change. %corresponding to $(f_{1},f_{2})$ tropicalizes as follows. 
The tropicalization on the remaining coordinates is as before, so we conclude that there is an induced cycle of length $-v(j(E))$ and thus the tropicalization is numerically faithful. It is not faithful however, since 
%The segment $\Gamma_{0}\Gamma_{1}$ is now contracted to a point and the length of the segment $\Gamma_{1}\Gamma_{2}$ is multiplied by $2$ by \cite[...]{BPR11}. The other segments are mapped isometrically onto their images. We thus conclude that the induced cycle has length $0+2/4v(q)+1/4v(q)+1/4v(q)=-v(j(E))$, which implies that this tropicalization is numerically faithful. It is not faithful however, since 
one of the segments is contracted to a point and another is expanded.  %It contains a cycle with three pieces %tropicalization
 %In terms of semistable models, this map just takes a type-$1$ point to the component it reduces to     
%We will abstractly prove the existence of a numerically faithful tropicalization that is not faithful. Using the techniques of the upcoming sections, it is possible to make these constructions concrete in terms of four-torsion points on elliptic curves with bad reduction. 

%We will give a tropicalization that 
\end{exa}

\section{Minimal models over non-Noetherian valuation rings}\label{ModelsSection}

In this section, we give a reduction theory for elliptic curves over $K$, similar to the one studied in \cite[Chapter VII]{Silv1}. For simplicity, we will assume that $\mathrm{char}(k)\neq{2,3}$. In the general case, one can still write down minimal models using a variant of Tate's algorithm, see \cite[Chapter IV, Algorithm 9.4]{silv2}. In the discretely valued case, the most convenient way to study the reduction type of the elliptic curve $E$ is through its N\'{e}ron model $\mathcal{E}$. %We refer the reader to [...] for more background regarding N\'{e}ron models and group schemes. 
Since we are in the non-discrete case however, we cannot use this machinery. Furthermore, there is no direct generalization of N\'{e}ron models to the non-discrete case available at the present, so we will study the reduction of $E/K$ using the theory of minimal models, which does generalize to the non-discrete case. 

Let $E/K$ be an elliptic curve, as defined in \cite[Chapter III]{Silv1}. Using the Riemann-Roch theorem, one can show that every such elliptic curve can be described by a Weierstrass equation:
\begin{equation}
y^2+a_{1}xy+a_{3}y=x^3+a_{2}x^2+a_{4}x+a_{6},
\end{equation}
see \cite[Chapter III, Proposition 3.1]{Silv1}. By applying an appropriate scaling transformation, one can then assume that $v(a_{i})\geq{0}$ for every $i$. We will call such Weierstrass equations {\it{integral Weierstrass models}} or integral Weierstrass equations. For any such Weierstrass model, one obtains a reduced Weierstrass equation over $k$ by reducing the coefficients $a_{i}$ mod $\mathfrak{m}$. This reduced Weierstrass equation is not canonical in any sense however: two $K$-isomorphic models can lead to non-isomorphic reduced curves, as the following example shows.
\begin{exa}\label{NonCanonicalReductionExample}
Consider the integral Weierstrass equation 
\begin{equation}\label{NonCanonicalReduction}
W:y^2=x^3+\varpi^{4}x+\varpi^{6}.
\end{equation}
The reduced curve in this case is given by the equation
\begin{equation}\label{SingularCurveExample}
y^2=x^3,
\end{equation}
which defines a singular curve. We now consider an isomorphic curve whose reduction is nonsingular. Dividing by $\varpi^{6}$ on both sides of Equation \ref{NonCanonicalReduction} and taking $y'=\dfrac{y}{\varpi^{3}}$ and $x'=\dfrac{x}{\varpi^{2}}$, we obtain
\begin{equation}\label{EquationMinimalModelNonSingular}
y'^{2}=x'^3+x'+1.
\end{equation}
Reducing the coefficients mod $\mathfrak{m}$ then yields a nonsingular curve, which consequently is not isomorphic to the reduced curve in Equation \ref{SingularCurveExample}. 
\end{exa} 
We are thus led to impose an additional condition on the models over $R$ to ensure some kind of canonicity. The notion we will be using is that of a {\it{minimal model}}, as in \cite[Chapter VII]{Silv1}. 
\begin{mydef}
\label{MinimalModelDefinition}
Let $E/K$ be an elliptic curve with integral Weierstrass equation
\begin{equation}
W:y^2+a_{1}xy+a_{3}y=x^3+a_{2}x^2+a_{4}x+a_{6}.
\end{equation}
and discriminant $\Delta\in{K}$ (see \cite[Chapter III, Page 42]{Silv1} for an explicit formula). 
Then $W$ is said to be {\it{minimal}} if $v(\Delta)$ is minimal among all integral Weierstrass equations for $E$. We refer to such a model as a {\it{minimal Weierstrass model}} and we denote it by $W/R$. 
\end{mydef}

\begin{exa}
Let $W$ be the Weierstrass equation given by $y^2=x^3+\varpi^{4}x+\varpi^{6}$ as in Example \ref{NonCanonicalReduction}. Then the discriminant of $W$ is given by $\Delta_{W}=-496\cdot{\varpi^{12}}$ and  $v(\Delta_{W})=12$. The Weierstrass equation $W'$ in Equation 
\ref{EquationMinimalModelNonSingular} is also a Weierstrass equation for the same elliptic curve, but it has $v(\Delta_{W'})=0$, so $W$ is not minimal. Note that $W'$ is automatically minimal, since the valuation of the discriminant cannot become any smaller for integral Weierstrass equations. 
\end{exa}

Using our assumption on the residue field, we now give the following convenient criterion for an integral Weierstrass model $W$ to be minimal. First recall that for any field of characteristic not equal to $2$ or $3$, any minimal Weierstrass model is isomorphic to one of the form
\begin{equation}
y^2=x^3-27c_{4}x-54c_{6}.
\end{equation}
Indeed, the transformations
\begin{align}
y\longmapsto{\dfrac{1}{2}(y-a_{1}x-a_{3})}
\end{align}
and 
\begin{equation}
(x,y)\longmapsto{(\dfrac{x-3b_{2}}{36},\dfrac{y}{108})}
\end{equation}
on \cite[Pages 42 and 43]{Silv1} are invertible over $R$ and the valuations of their discriminant are the same by the tables on \cite[Page 45]{Silv1}.  %see \cite[...]{Silv1}.

\begin{lma}\label{CriterionMinimality}
Let $W/R$ be a Weierstrass equation for an elliptic curve $E/K$. Then $W/R$ is a minimal model if and only if 
\begin{equation}
\mathrm{min}\{v(c_{4}),v(c_{6})\}=0.
\end{equation}
\end{lma}
\begin{proof}
Suppose that either $v(c_{4})=0$ or $v(c_{6})=0$ and suppose for a contradiction that there exists an integral Weierstrass equation $W'/R$ with $v(\Delta')<v(\Delta)$. We then have
\begin{align}\label{EquationMinimal}
u^{4}\cdot{}c'_{4}&=c_{4}\\
u^{6}\cdot{}c'_{6}&=c_{6}.
\end{align}
for a standard transformation relating $W'$ and $W$ as in \cite[Page 44]{Silv1} (every isomorphism is of this form by \cite[Chapter III, Proposition 3.1]{Silv1})). But $v(u)>0$ (since $v(\Delta')<v(\Delta)$), so either $v(c'_{4})<0$ or $v(c'_{6})<0$, a contradiction. Note that this proof doesn't use the assumption on the characteristic.

Suppose now that $W/R$ is a minimal model, which we can assume to be of the form
\begin{equation}
y^2=x^3-27c_{4}x-54c_{6}
\end{equation}
by our assumption on the residue characteristic. Suppose that $v(c_{4}),v(c_{6})>0$ and consider 
\begin{equation}
m:=\mathrm{min}\{v(c_{4})/4,v(c_{6})/6\}.
\end{equation}
Let $u$ be any element with valuation $m$ (which exists because $K$ is algebraically closed) and consider the transformation
\begin{align*}
x&=u^2\cdot{}x',\\
y&=u^3\cdot{}y'.
\end{align*}
By Equation \ref{EquationMinimal}, we see that $v(c'_{4}),v(c'_{6})\geq{0}$ (in fact, one of them has to be zero) and by 
\begin{equation}
u^{12}\Delta'=\Delta,
\end{equation}
we see that $v(\Delta')<v(\Delta)$, a contradiction. This proves the lemma.
\end{proof}
\begin{rem}\label{RemarkMinimality}
Note that the proof of Lemma \ref{CriterionMinimality} gives an explicit way of determining a minimal Weierstrass model: we take any integral equation and determine the $c_{i}$. By applying the transformation in the proof, we then immediately obtain a minimal model. % Of course, this then also immediately settles the matter of the existence of a minimal model for any elliptic curve.  %which furthermore automatically proves the existence of such a model. % minimal model for any elliptic curve $E/K$.  %In particular, we see that a minimal model exists for any elliptic curve $E/K$. 
\end{rem}
\begin{rem}
We note that there exist minimal Weierstrass models over valued fields with residue characteristic $2$ and $3$ with $v(c_{4})>0$ and $v(c_{6})>0$. As an example, let $K=\mathbb{Q}_{2}$ be the field of $2$-adic numbers and consider the elliptic curve given by 
\begin{equation}
y^2+y=x^3.
\end{equation}
This curve has good reduction, so $v(\Delta)=0$. We then have $c_{4}=0$ and $c_{6}=-216$, so both of the invariants have strictly positive valuation. 
\end{rem}

\begin{prop}\label{MinimalModelsBasic}
Let $E/K$ be an elliptic curve. Then the following hold.
\begin{itemize}
\item $E$ has a minimal Weierstrass model $W/R$. 
\item A minimal Weierstrass model is unique up to a change of coordinates
\begin{align}\label{CoordinateChange}
x&=u^2x'+r\\
y&=u^3y'+u^2sx'+t,
\end{align}
with $u\in{R^{*}}$ and $r,s,t\in{R}$.
\item Let $W/R$ be an integral Weierstrass equation. Then any change of coordinates \begin{align}
x&=u^2x'+r\\
y&=u^3y'+u^2sx'+t
\end{align}
that turns $W/R$ into a minimal Weierstrass model $W'/R$ satisfies $u,r,s,t,\in{R}$. 
%as above used to produc
\end{itemize}
\end{prop}
\begin{proof}
The first part follows from Lemma \ref{CriterionMinimality} and Remark \ref{RemarkMinimality}.  The second and third part follow in exactly the same way as in \cite[Chapter VII, Proposition 1.3]{Silv1}. Note that the proofs of these parts do not use the discreteness of the valuation, nor the completeness of $K$. We leave it to the reader to fill in the details. 
%The first part follows from Lemma \ref{CriterionMinimality} and Remark \ref{RemarkMinimality}. The second and third part follow in exactly the same way as in \cite[Chapter VII, Proposition 1.3]{Silv1}. We leave the details to the reader. 
\end{proof}

Let $W/R$ be a minimal Weierstrass model given by
\begin{equation}
y^2+a_{1}xy+a_{3}y=x^3+a_{2}x^{2}+a_{4}x+a_{6},
\end{equation}
where $a_{i}\in{R}$. Using the natural reduction map $R\rightarrow{R/\mathfrak{m}=k}$, we can then consider the reduced Weierstrass equation
\begin{equation}
y^2+\overline{a}_{1}xy+\overline{a}_{3}y=x^3+\overline{a}_{2}x^{2}+\overline{a}_{4}x+\overline{a}_{6}.
\end{equation}

By Proposition \ref{MinimalModelsBasic}, any two minimal Weierstrass models $W$ and $W'$ are related by a coordinate change as in Equation \ref{CoordinateChange} with $u\in{R^{*}}$ and $r,s,t,\in{R}$. Reducing this coordinate change $\bmod{\mathfrak{m}}$, we obtain a standard coordinate change over the residue field. We thus see that any two minimal Weierstrass models give rise to isomorphic reduced curves. The reduced equation is thus independent of the minimal Weierstrass model, up to standard coordinate changes over the residue field. Note that the notion of minimality is crucial here, as the reduced curve of a non-minimal Weierstrass equation can be non-isomorphic to the reduced curve of a minimal Weierstrass equation, see Example \ref{NonCanonicalReductionExample}. The reduced curve associated to any minimal Weierstrass model $W/R$ will be denoted by $\overline{E}/k$ or $\overline{E}$.

We now give a reduction map $E(K)\rightarrow{\overline{E}(k)}$ in terms of projective coordinates. Write $P$ as $P=[x_{0},y_{0},z_{0}]$. By scaling these coordinates, we can find an equivalent triple such that at least one of $x_{0},y_{0},z_{0}$ is a unit. The reduced point
\begin{equation}
\overline{P}=[\overline{x}_{0},\overline{y}_{0},\overline{z}_{0}]
\end{equation}
then lies in $\overline{E}(k)$. This gives us the reduction map
\begin{equation}
E(K)\rightarrow{\overline{E}(k)}, \qquad{P}\mapsto{\overline{P}}.
\end{equation}

We now recall some facts regarding the reduced curve $\overline{E}/k$. This curve is singular if and only if $\overline{\Delta}=0$, where $\Delta$ is the discriminant associated to the Weierstrass equation. Furthermore, any Weierstrass equation can have only one singularity, which is either a cusp or a node. If it has a singularity, then we can put a group structure on the smooth points by \cite[Chapter III, Proposition 2.5]{Silv1}. There are three possible singularities and for each one we have a different group structure. We can in fact characterize the type of singularity we get by reducing the invariants $c_{i}$, as the following proposition shows. %We summarize the reduction  This group structure is uniquely determined by the kind of singularity, which is the content of the next proposition. % and the corresponding groups are characterized by the following proposition. %(with $c_{4}=0$) or a node with $c_{4}\neq{0}$. We then have
\begin{prop}\label{ReductionType}
Let $W/R$ be a minimal Weierstrass model for an elliptic curve $E/K$ with reduced curve $\overline{E}$. Let $\overline{E}_{\mathrm{ns}}$ be the set of nonsingular points. Then the following hold:
\begin{itemize}
\item $\overline{E}$ is an elliptic curve if and only if $v(\Delta)=0$. We have $\overline{E}=\overline{E}_{\mathrm{ns}}$. In this case, the elliptic curve $E$ is said to have {\it{good reduction}}.
\item $\overline{E}$ has a cusp if and only if $v(\Delta)>0$ and $\overline{c}_{4}=0$. We have $\overline{E}_{ns}\simeq{k^{+}}$, the additive group of the residue field $k$. In this case, the elliptic curve is said to have {\it{additive reduction}}. 
\item $\overline{E}$ has a node if and only if $v(\Delta)>0$ and $\overline{c}_{4}\neq{0}$. We have $\overline{E}_{ns}\simeq{k^{*}}$, the multiplicative group of the residue field $k$. In this case, the elliptic curve is said to have {\it{multiplicative reduction}}.

\end{itemize}
\end{prop}
\begin{proof}
The proof of \cite[Chapter VII, Proposition 5.1]{Silv1} still works in the non-discrete case, as one can easily check. We leave the details to the reader. 
\end{proof}

\begin{exa}\label{ExampleMultiplicative}
Let $E$ be the elliptic curve defined by the Weierstrass minimal model
\begin{equation}
y^2=x^3+x^2+\varpi^2,
\end{equation}
where $v(\varpi)>0$. The reduced curve is then given by
\begin{equation}
y^2=x^3+x^2,
\end{equation}
which has the singularity $(0,0)$. Note that this singularity is a node, so $E$ has multiplicative reduction. 
\end{exa}

We now relate the reduction type of an elliptic curve to the valuation of the $j$-invariant. To that end, we will need the following formula:
\begin{equation}
j=\dfrac{c_{4}^{3}}{\Delta}=1728\dfrac{c_{4}^3}{c_{4}^3-c_{6}^2},
\end{equation}
see \cite[Chapter III, Page 42]{Silv1}.

\begin{prop}\label{ReductionTypeInvariantProposition}
Let $E/K$ be an elliptic curve with minimal model $W/R$. Then:
\begin{itemize}
\item $E$ has good reduction if and only if $v(j)\geq{0}$,
\item $E$ has multiplicative reduction if and only if $v(j)<{0}$. 
\end{itemize}
In particular, we see that $E$ cannot have additive reduction.
\end{prop}

\begin{proof}
Suppose that $E$ has good reduction. Then $v(\Delta(E))=0$ and consequently $v(j)=3v(c_{4})-v(\Delta)\geq{0}$, as desired. Suppose that $v(j)\geq{0}$ and let $W/R$ be a minimal model of the form
\begin{equation}
y^2=x^3-27c_{4}x-54c_{6}.
\end{equation}
Suppose that $v(\Delta)>0$. Then we must have $3v(c_{4})=2v(c_{6})$.
But by Lemma \ref{CriterionMinimality}, we see that either $v(c_{4})=0$ or $v(c_{6})=0$, so $v(c_{4})=0=v(c_{6})$. But then $v(j)=3v(c_{4})-v(\Delta)<0$, a contradiction. We conclude that $v(\Delta)=0$ and thus $E$ has good reduction.

Suppose now that $v(j)<0$. Then we must have $v(\Delta)>0$. Suppose that $E$ has additive reduction. Then $v(c_{4})>0$ by Proposition \ref{ReductionType} and consequently $v(c_{6})=0$ by Lemma \ref{CriterionMinimality}. But then $v(\Delta)=0$, a contradiction. We conclude that $E$ has multiplicative reduction, as desired. Suppose that $E$ has multiplicative reduction. By what was proved earlier, we cannot have $v(j)\geq{0}$, so we must have $v(j)<0$. This concludes the proof.

\end{proof}

Consider the following subset of $E(K)$:%We now subdivide $E(K)$ into two sets: we consider the points that reduce to nonsingular points and the points that reduce to singular points. The points that reduce to nonsingular points turn out to be a subgroup of %That is:
\begin{equation}
E_{0}(K)=\{P\in{E(K):\overline{P}\in{\overline{E}_{\mathrm{ns}}(k)}}\}.
\end{equation}
By \cite[Chapter VII, Proposition 2.1]{Silv1} (note that the proof only uses the fact that $R$ is Henselian), we find that $E_{0}(K)$ is a subgroup of $E(K)$ and we have an exact sequence 
\begin{equation}\label{ExactSequence}
0\rightarrow{E_{1}(K)}\rightarrow{}E_{0}(K)\rightarrow{}\overline{E}_{\mathrm{ns}}(k)\rightarrow{}0.
\end{equation}
Here $E_{1}(K)$ is the kernel of the reduction map, i.e.
\begin{equation}
E_{1}(K)=\{P\in{E(K)}:\overline{P}=\overline{\mathcal{O}}\},
\end{equation}
where $\mathcal{O}$ is the point at infinity. Note that the projective point $[0,1,0]$ is always nonsingular for any Weierstrass equation (see \cite[Chapter III, Proposition 1.2]{Silv1}), so we have $E_{1}(K)\subseteq{E_{0}(K)}$. Since $E_{0}(K)$ is a subgroup, we can consider the quotient
\begin{equation}
E/E_{0}(K):=E(K)/E_{0}(K).
\end{equation}
A point $P\in{E(K)}$ then gives rise to a nontrivial point in $E/E_{0}(K)$ if and only if $P$ reduces to a singular point. 
\begin{exa}\label{ExampleMinimalModel}
Let $E$ again be the elliptic curve defined by the Weierstrass minimal model
\begin{equation}
y^2=x^3+x^2+\varpi^2,
\end{equation}
where $v(\varpi)>0$. We saw in Example \ref{ExampleMultiplicative} that $E$ has multiplicative reduction. We now give an example of a nontrivial point in $E/E_{0}(K)$, $E_{0}(K)$ and $E_{1}(K)$ respectively.  %and thus that $\overline{E}%. The reduced curve is then given by
%\begin{equation}
%y^2=x^3+x^2,
%\end{equation}
%which has the singularity $(0,0)$. Note that this singularity is a node, so $E$ has multiplicative reduction. 
Consider the point 
\begin{equation}
P_{1}=(0,\varpi)\in{E(K)}.
\end{equation}
Since $P_{1}$ reduces to the singular point, we see that $P_{1}\notin{E_{0}(K)}$. Let $\alpha$ be a square root of ${2+\varpi^{2}}$ in $K$ and consider the point $P_{2}=(1,\alpha)$. Its reduction is then given by $\overline{P}_{2}=(\overline{1},\overline{\alpha})\in{\overline{E}_{\mathrm{ns}}}$, so we find that $P_{2}\in{E_{0}(K)}$. Lastly, let $\beta$ be a square root of ${1+\varpi^2+\varpi^8}$ and consider the projective point $P_{3}=[\varpi,\beta,\varpi^3]$. In terms of $x$ and $y$ coordinates, this is given by
\begin{equation}
P_{3}=(\dfrac{1}{\varpi^2},\dfrac{\beta}{\varpi^3}).
\end{equation}
Since $\beta\notin\mathfrak{m}$, we find that $P_{3}\in{E_{1}(K)}$.%(1/\varpi^2,\dfrac{}

 %Note that $E$ has multiplicative reduction, as the singularity

\end{exa}

We now have two subgroups of $E$ at our disposal: $E_{1}$ and $E_{0}$. We are interested in the torsion structure of these groups. That is, for any abelian group $G$ and integer $n>1$, we consider the subgroup
\begin{equation}
G[n]=\{g\in{G}:{n}\cdot{g}=e\},
\end{equation} 
where $e$ is the identity of $G$. For elliptic curves, we will denote the $n$-torsion subgroup of the $K$-valued points by $E[n](K)$. For $K$ algebraically closed and $n$ coprime to $\mathrm{char}(K)$, we then have  %:=E(K) 
%For any elliptic curve $E$ over an algebraically closed field and any integer $n$ coprime to $\mathrm{char}(K)$, we have
\begin{equation}
E[n](K)=(\mathbb{Z}/n\mathbb{Z})^{2}
\end{equation}
see \cite[Chapter III, Corollary 6.4]{Silv1}. We are especially interested in the case $n=3$. In this case, the torsion points have a very geometric flavor to them, as the following well-known lemma shows:
\begin{lma}{\bf{[Inflection points]}}
\label{InflectionPoint}
Let $E/K$ be an elliptic curve with a point $P\in{E(K)}$. Then $P\in{E[3](K)}$ if and only if $P$ is an {\it{inflection point}}. That is, the tangent line at $P$ only intersects $E$ at $P$.  
\end{lma} 
\begin{proof}
Suppose that the tangent line is given by $H(X,Y,Z)=\alpha{X}+\beta{Y}+\gamma{Z}=0$ and that it only intersects $E$ at $P$. By Bezout's theorem  applied to $E\subset{\mathbb{P}^{2}}$, it intersects $E$ triply. The divisor of $H/Z$ is then $3(P)-3(\mathcal{O})$ and thus $P$ is a point of order three, as desired. 
Conversely, let $P$ be a point of order three and let $H(X,Y,Z)$ be the tangent line at $P$. %suppose that the tangent line with equation $H(X,Y,Z)=0$ at $P$ intersects $E$ at another point $Q$. 
Then $\mathrm{div}(H/Z)=2(P)+(Q)-3(\mathcal{O})$ for some $Q\in{E(K)}$ and consequently the degree zero divisor $(Q)-(\mathcal{O})$ is the inverse of $2(P)-2(\mathcal{O})$ in $\mathrm{Pic}^{0}(E)$. But this inverse is exactly $(P)-(\mathcal{O})$, so we find $P=Q$, as desired. 
\end{proof}

\begin{exa}\label{ExampleInflectionFamily}
Let $E$ be the elliptic curve defined by the affine (minimal) Weierstrass equation 
\begin{equation}
y^2=x^3+(x-\varpi)^2
\end{equation}

\begin{figure}[h!]
\centering
\includegraphics[width=6cm]{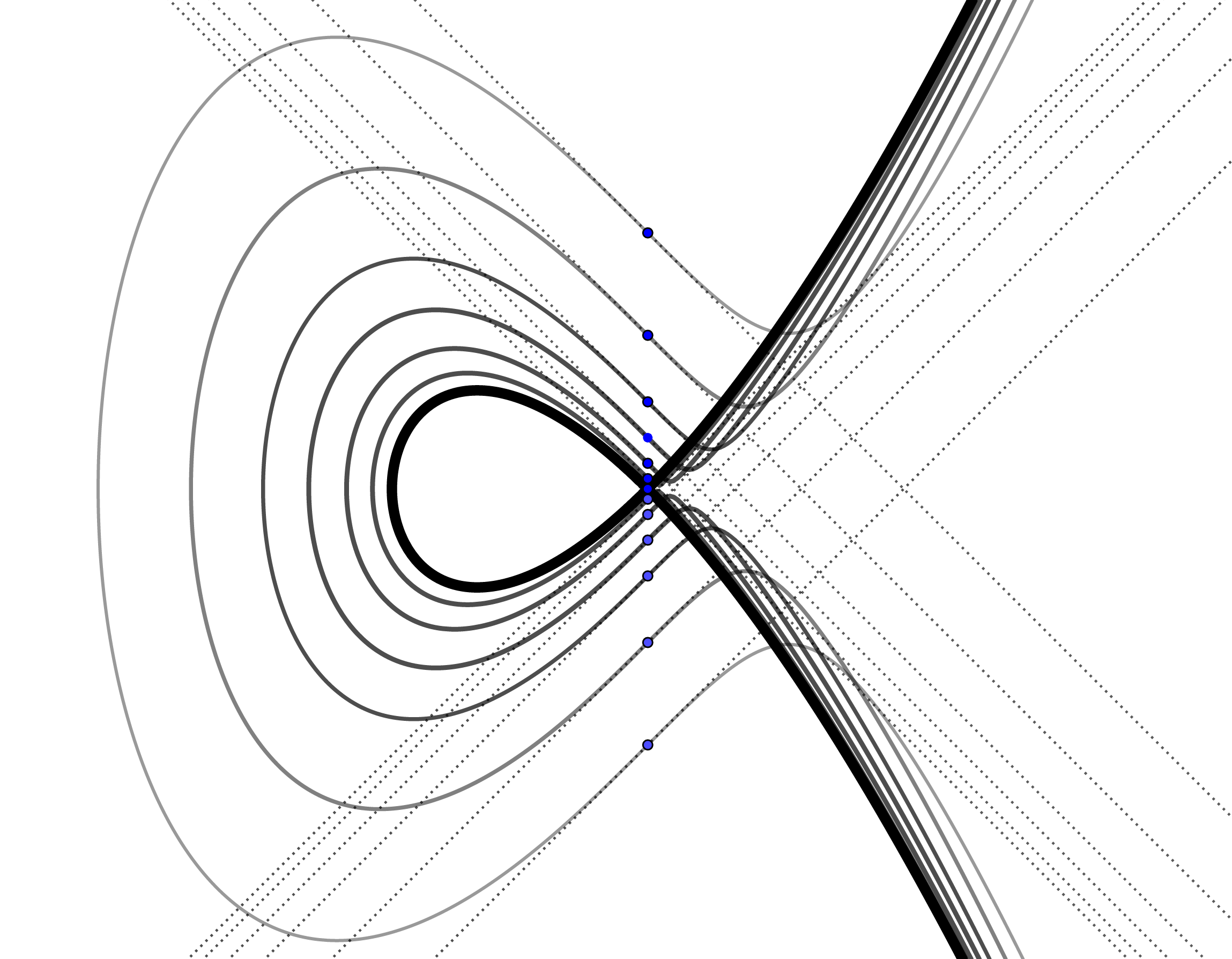}
\caption{\label{PlaatjeDegTrop4} The family of elliptic curves in Example \ref{ExampleInflectionFamily}. The blue points indicate inflection points on the members in this family. The tangent lines at these points intersect the curve only at that inflection point. }%As the curves get closer to the nodal curve, the tangent lines converge to the two distinct tangent directions of the singular curve, so this defines a nontrivial point in $E/E^{0}(K)[3]$. } %and a family of blue marked inflection points. The  %Every member in the family has two blue {\it{inflection points}} for which the corresponding dotted tangent lines intersect the curve only at that point. %the rest of the curve only at that point. 

%at which the tangent line  %The idea of our proof of Theorem \ref{MainTheorem} %The proof of Theorem \ref{MainTheorem} follows  }%
\end{figure}
This contains the rational point $P=(0,\varpi)$ and we claim that this is an inflection point. Indeed, rewriting the equation yields
\begin{equation}
(y-(x-\varpi))(y+(x-\varpi))=x^3.
\end{equation}
We then have that 
\begin{equation}
\mathrm{div}(y+(x-\varpi))=3(P)-3\mathcal{O},
\end{equation}
so $P$ is a three-torsion point because the group structure on $E$ is induced by that on $\mathrm{Pic}^{0}(E)$, see \cite[Chapter III, Proposition 3.4(d)]{Silv1}.  By Lemma \ref{InflectionPoint}, we then see that $P$ is indeed an inflection point. The real picture for this family can be found in Figure \ref{PlaatjeDegTrop4}. Here we evaluated $\varpi$ at real numbers close to zero. A similar family will be used in Section \ref{FinalSection} for the proof of Theorem \ref{MainTheorem}.  %For a real description 

\end{exa}

We will use this characterization of inflections points in Lemma \ref{InflectionPoint}  in the proof of Lemma \ref{ExplicitFamily1} to give an explicit family of elliptic curves with a marked inflection point on each member. This explicit form then allows us to find the desired tropicalization we are after.  %5, see Lemma \ref{ExplicitFamily1}. 

 We now consider the problem of determining how the $n$-torsion of an elliptic curve is distributed over $E_{1}$, $E_{0}$ and the quotient $E/E_{0}$. To that end, we first consider the $n$-torsion of $E_{1}$. %We now consider the $n$-torsion o$E_{1}$. 
By \cite[Chapter VII, Proposition 2.2]{Silv1}, there is an isomorphism
\begin{equation}
\hat{E}(\mathfrak{m})\rightarrow{E_{1}(K)},
\end{equation}
where $\hat{E}$ is the formal group associated to $E$ (see \cite[Chapter IV]{Silv1}) and $\mathfrak{m}$ the maximal ideal of $R$. Since the multiplication by $n$ map is invertible on $\hat{E}$ for any $n$ coprime to the residue characteristic, we obtain the following

 %Using this additional structure, we obtain  %Since multiplication by $n$ %Using this, we easily obtain
\begin{prop}\label{NoTorsion}
Let $E/K$ be an elliptic curve and let $n$ be an integer that is coprime to the characteristic of the residue field $k$. Then $E_{1}(K)[n]=(0)$. % does not contain any $n$-torsion points.
\end{prop}
\begin{proof}
By \cite[Chapter IV, Proposition 2.3.b]{Silv1}, the multiplication by $n$-map $\hat{E}\rightarrow{\hat{E}}$ is an isomorphism. This directly implies that $\hat{E}(\mathfrak{m})[n]=(0)$, as desired.

\end{proof} 

\begin{exa}
Consider again the elliptic curve defined by
\begin{equation}
y^2=x^3+x^2+\varpi^{2}.
\end{equation}
We saw in Example \ref{ExampleMinimalModel} that $P_{3}=(\dfrac{1}{\varpi^{2}},\dfrac{\beta}{\varpi^{3}})\in{E_{1}(K)}$, where $\beta^2=1+\varpi^{2}+\varpi^{8}$. If we assume that $\mathrm{char}(k)=0$, then using Proposition \ref{NoTorsion}, we see that $P_{3}$ cannot be a torsion point, so $P_{3}$ has infinite order. This trick can be used more generally to create points of infinite order on families of elliptic curves.  
%We saw in Example \ref{ExampleMinimalModel} that $P_{3}=(\dfrac{1}{\varpi^{2}},\dfrac{\beta}{\varpi^{3}})\in{E_{1}(K)}$, where $\beta^2=1+\varpi^{2}+\varpi^{8}$. Using Proposition \ref{NoTorsion}, we see that $P_{3}$ cannot be a torsion point, so $P_{3}$ has infinite order. This trick can be used more generally to create points of infinite order on families of elliptic curves.  
%Point with infinite order from Example \ref{ExampleMinimalModel}. 
\end{exa}

%We now have four groups to our disposal: $E %It will turn  %This proposition will allow us to 
\begin{lma}\label{TorsionPointReduction}
Suppose that $E$ has multiplicative reduction with singular point $x\in{\overline{E}(k)}$ and let $n$ be coprime to the characteristic of the residue field. Then there exists a point $P\in{E(K)}$ of order $n$ such that $\overline{P}=x$.
\end{lma}
\begin{proof}
Suppose that every point $P$ of order $n$ of $E$ reduces to a nonsingular point. Then $E[n](K)\subset{E_{0}(K)}$. By Proposition \ref{NoTorsion} and the exact sequence from Equation \ref{ExactSequence}, we see that $E[n](K)$ injects into $\overline{E}_{\mathrm{ns}}(k)$ under the reduction map. %  By Lemma ..., the $n$-torsion of $E$ injects into $\overline{E}_{\mathrm{sm}}[n]$. 
But this is impossible: $\overline{E}_{\mathrm{sm}}[n]\simeq{k^{*}[n]\simeq{\mathbb{Z}/n\mathbb{Z}}}$ has order $n$, whereas $E[n](K)\simeq{(\mathbb{Z}/n\mathbb{Z})^{2}}$ has order $n^2$. We conclude that there exists a $P$ of order $n$ reducing to the singular point, as desired.  %\simeq{\mathbb{Z}/n\mathbb{Z}}$, which is the  
\end{proof}

\begin{exa}
Let $E$ be the curve in Example \ref{ExampleInflectionFamily}. The inflection point $P=(0,\varpi)$ then reduces to the singular point $(0,0)$ on the reduction $y^2=x^3+x^2$. We thus see that $P$ defines a nontrivial element in $E/E^{0}(K)[3]$, as guaranteed by Lemma \ref{TorsionPointReduction}. 
\end{exa}

\begin{rem}
Let us assume that $K$ is complete. Using the analytic uniformization theorem for elliptic curves with split multiplicative reduction (see Equation \ref{UniformizationEquation}), it is now much easier to obtain the point $P$ in this Lemma. Indeed, for any such elliptic curve $E/K$ with $v(j(E))<0$, one considers the analytic isomorphism 
\begin{equation}
E(K)\rightarrow{}K^{*}/\langle{q}\rangle,%\rightarrow{}
\end{equation}
where $q\in{K}$ is such that $-v(j(E))=v(q)>0$. To find the point $P$ as in Lemma \ref{TorsionPointReduction}, one simply takes $P=q^{1/n}$. We invite the reader to compare this to the material in Section \ref{FaithfulTropSection} and the construction in \cite[Theorem 6.2]{BPR11}. 
%We note that using the analytic uniformization theorem for elliptic curves with multiplicative reduction, it is much easier to obtain the point $P$ in this Lemma. For any such elliptic curve $E/K$ with $v(j)<0$, one considers the analytic isomorphism 
%\begin{equation}
%E(K)\rightarrow{}K^{*}/\langle{q}\rangle,%\rightarrow{}
%\end{equation}
%where $q\in{K}$ is such that $-v(j)=v(q)>0$. To find the point $P$ as in Lemma \ref{TorsionPointReduction}, one simply takes $P=q^{1/n}$. We invite the reader to compare this with the construction in \cite[Theorem 6.2]{BPR11}.   %This is exactly what is done in [Baker].
\end{rem}

\section{Creating numerically faithful tropicalizations using minimal models}\label{FinalSection}

In this section, we will show that any elliptic curve $E/K$ with $v(j(E))<0$ admits a numerically faithful tropicalization as in Definition \ref{NumericallyFaithfulDefinition}. Using the criterion in \cite[Corollary 5.28(2)]{BPR11} and Example \ref{FaithfulTrop3}, we then also see that this particular embedding defines a faithful tropicalization. %define the notion of a faithful tropicalization of an elliptic curve over $K$ with $v(j)<0$. We then show that any such elliptic curve $E/K$ with $v(j)<0$ admits a faithful tropicalization.
 To find this numerically faithful tropicalization, we use a three torsion point $P$ that reduces to the singular point of the reduced curve corresponding to a Weierstrass minimal model, which exists by Lemma \ref{TorsionPointReduction}. We then construct two principal divisors $\mathrm{div}(f),\mathrm{div}(g)\in{\mathrm{Prin}(E)}$ using this torsion point. These two principal divisors then give rise to a closed embedding 
%This point then gives two functions $f,g\in{K(E)}$ that give rise to a closed  embedding
\begin{equation}
E\rightarrow{\mathbb{P}^{2}},
\end{equation}
whose
tropicalization is easily shown to contain a cycle of length $-v(j(E))$, which concludes the proof. The idea of the proof is mostly based on the construction in \cite[Theorem 6.2]{BPR11}. In that proof, they abstractly show using the Poincar\'{e}-Lelong formula that the corresponding piecewise linear functions $-\mathrm{log}|f|$ and $-\mathrm{log}|g|$ separate the points of the Berkovich minimal skeleton of $E$. %We will show how this idea can be used to create other faithful tropicalizations.

Let $E$ be an elliptic curve with $v(j(E))<0$. By Proposition \ref{ReductionTypeInvariantProposition}, this implies that $E$ has multiplicative reduction. By Lemma \ref{TorsionPointReduction}, there exists a $P\in{E[3](K)}$ such that $P$ reduces to the singular point of a minimal Weierstrass model. In other words, the class of $P$ in $E/E_{0}(K)$ is nontrivial. Let $W$ be a minimal Weierstrass model for $E$ of the form
\begin{equation}
y^2=x^3+b_{2}x^2+b_{4}x+b_{6}.
\end{equation}
Since $P$ reduces to the singular point on $\overline{E}$, we see in particular that $v(x(P))\geq{0}$. The transformation
\begin{comment}
\begin{lemma}\label{Nagell}
Let $E$ be an elliptic curve with multiplicative reduction, let $W$ be a minimal Weierstrass model and let  $P$ be a three-torsion point. Then $v(x(P))\geq{0}$.
\end{lemma}
\begin{proof}
Suppose for a contradiction that $v(x(P))<0$. Then $P$ reduces to $\overline{\mathcal{O}}$ and thus $P\in{E_{1}(K)}$. But $E_{1}(K)[3]=(0)$ by Lemma \ref{NoTorsion}, so we conclude that $v(x(P))\geq{0}$, as desired. % contains no three torsion points by 
\end{proof}
Let $W$ be a minimal Weierstrass model as above. By Lemma \ref{Nagell}, we see that the transformation 
\end{comment}
$x\mapsto{x-x(P)}$ then transforms $W$ into another {\it{integral}} Weierstrass model, which is again minimal by Proposition \ref{MinimalModelsBasic}. We will again denote this minimal Weierstrass model by $W$. In this new model, we have $x(P)=0$. % Consider the translation $%By translating the $x$-coordinate over $x(P)$, we see %we can assume that $x(P)=0$.
\begin{lma}\label{ExplicitFamily1}
Let $E$, $W$ and $P$ be as above. Then $\Delta(b_{2}x^2+b_{4}x+b_{6})=b_{4}^2-4b_{2}b_{6}=0$. In other words, we can write 
\begin{equation}
y^2=x^3+a(x-b)^2
\end{equation}
for $a$ and $b$ in $R$.
\end{lma}
\begin{proof}
Let $P=(0,y(P))$ be the $3$-torsion point reducing to the singular point on $\overline{E}$ as above.  By Lemma \ref{InflectionPoint}, it is an inflection point: its tangent line intersects $E$ only at $P$.  %We will use the following criterion for a point to be of order three: $P\in{E[3](K)}$ if and only if the tangent line at $P$ intersects $E$ only at $P$. The proof of the criterion goes as follows. Suppose that the tangent line is given by $H(X,Y,Z)=\alpha{X}+\beta{Y}+\gamma{Z}=0$ and that it only intersects $E$ at $P$. By Bezout, it intersects $E$ triply. The divisor of $H/Z$ is then $3(P)-3(\mathcal{O})$ and thus $P$ is a point of order three, as desired. 
%Let $P$ be a point of order three and let $H(X,Y,Z)$ be the tangent line at $P$. %suppose that the tangent line with equation $H(X,Y,Z)=0$ at $P$ intersects $E$ at another point $Q$. 
%Then $\mathrm{div}(H/Z)=2(P)+(Q)-3(\mathcal{O})$ for some $Q\in{E(K)}$ and consequently the degree zero divisor $(Q)-(\mathcal{O})$ is the inverse of $2(P)-2(\mathcal{O})$ in $\mathrm{Pic}^{0}(E)$. But this inverse is exactly $(P)-(\mathcal{O})$, so we find $P=Q$, as desired. 
%We now continue the proof of the lemma. Let $P=(0,y(P))$ be the point of order three reducing to the singular point on $\overline{E}$. 
The tangent line at $P$ is given by the equation
\begin{equation}
\dfrac{\partial{f}}{\partial{x}}(P)\cdot{x}+\dfrac{\partial{f}}{\partial{y}}(P)(y-y(P))=0,
\end{equation}
where $f=y^2-(x^3+b_{2}x^2+b_{4}x+b_{6})$. 
We thus obtain
\begin{equation}
-b_{4}x+2y(P)(y-y(P))=0. 
\end{equation}
In terms of $y$, we obtain
\begin{equation}\label{Expression}
y=\dfrac{2b_{6}+b_{4}x}{2y(P)}.
\end{equation}
Note here that $y(P)\neq{0}$, since otherwise $P$ would be a $2$-torsion point. 

Squaring the last expression and equating it to $x^3+b_{2}x^2+b_{4}x+b_{6}$, we obtain the cubic equation
\begin{equation}
x^3+(b_{2}-\dfrac{b_{4}^{2}}{4b_{6}})x^{2}+h(x)=0,
\end{equation}
where $h(x)$ is some linear polynomial. Since $x=0$ is a triple root of this equation, we must have $b_{2}-\dfrac{b_{4}^{2}}{4b_{6}}=0$, or in other words:
\begin{equation}
b_{4}^2=4b_{2}b_{6}.
\end{equation}
This means that the discriminant of the quadratic form $b_{2}x^2+b_{4}x+b_{6}$ is zero and we can thus write it as $a(x-b)^{2}$ for some $a$ and $b$ in $R$. This concludes the proof of the lemma. 
\end{proof}

\begin{lma}\label{LemmaValuation}
Let $E$, $W$ and $P$ be as above. Then $v(a)=0$ and $v(b)>0$.
\end{lma}
\begin{proof}
If $v(a)>0$, then $E$ has additive reduction since the reduced curve is given in this case by $y^2=x^3$. This contradicts our assumption that $E$ has multiplicative reduction and we thus see that $v(a)=0$. For the second part, we will use our assumption that $P$ reduces to the singular point of $\overline{E}$. Let $a'$ be a square root of $a$. We then have %Then $P$ is given by either
\begin{equation}
P=(0,a'b) \text{ or }P=(0,-a'b).
\end{equation}
Without loss of generality, we can assume that $P=(0,a'b)$. Since $P$ reduces to the singular point on $\overline{E}$, we must have
\begin{equation}
(\dfrac{\partial{f}}{\partial{x}}(\overline{P}),\dfrac{\partial{f}}{\partial{y}}(\overline{P}))=(0,0),
\end{equation}
where $f=y^2-x^3-\overline{a}(x-\overline{b})^{2}\in{k[x,y]}$. But then $\overline{a'b}=0$ and consequently $\overline{b}=\overline{0}$, since $\overline{a'}\neq{\overline{0}}$ by $v(a)=0$. This proves that $v(b)>0$, as desired. 
\end{proof}

Let $W$ be given by
\begin{equation}
y^2=x^3+a(x-b)^2.
\end{equation}
We can rewrite this as
\begin{equation}
(y-a'(x-b))(y+a'(x-b))=x^3.
\end{equation}
We then quickly see that $\mathrm{div}(y-a'(x-b))=3P'-3\mathcal{O}$ and $\mathrm{div}(y+a'(x-b))=3P-3\mathcal{O}$, where $P'=(0,-a'b)$. Another calculation then shows that
\begin{equation}
\mathrm{div}(x)=P+P'-2\mathcal{O}.
\end{equation}
We now explicitly give the two principal divisors $f$ and $g$ that will give the desired embedding into $\mathbb{P}^{2}$ . We take
\begin{align}\label{PrincipalDivisor1}
f&=\dfrac{x^2}{y-a'(x-b)}\\
g&=\dfrac{x^2}{y+a'(x-b)}.\label{PrincipalDivisor2}
\end{align}
Using the above identities for the divisors of $x$ and $y\pm{a'(x-b)}$, we then obtain
\begin{align*}
\mathrm{div}(f)&=2P-P'-\mathcal{O}\\
\mathrm{div}(g)&=2P'-P-\mathcal{O}.
\end{align*}
We now explicitly calculate the closed embedding induced by $(f,g)$.

\begin{lma}\label{EmbeddingLemma}
Let $E$ be as above and let $f$ and $g$ be as in Equations \ref{PrincipalDivisor1} and \ref{PrincipalDivisor2}. Then the image in $\mathbb{P}^{2}$ of the embedding induced by $f,g$ is cut out by the affine equation
\begin{equation}
f^2g+2a'fg-fg^2-2a'b=0.
\end{equation}
\end{lma}
\begin{proof}
First note that $fg=x$, by virtue of 
\begin{equation}
(y-a'(x-b))(y+a'(x-b))=x^3.
\end{equation} We now express $y$ in terms of $f$ and $g$. Note that $(y-a'(x-b))\cdot{f}=x^2$, so
\begin{equation}
y=\dfrac{x^2}{f}+a'(x-b)=fg^2+a'(fg-b).
\end{equation}
Plugging in $x$ and $y$ in $y+a'(x-b)=\dfrac{x^2}{g}=f^2g$, we obtain
\begin{equation}
y+a'(x-b)=fg^2+a'(fg-b)+a'(fg-b)=f^2g
\end{equation} 
and thus
\begin{equation}
fg^2+2a'fg-2a'b-f^2g=0,
\end{equation}
as desired.
%Since $fg^2+2a'fg-2a'b-f^2g$ is irreducible in $K[f,g]$, we see that the image is indeed cut out by this affine equation, as desired. 

\end{proof}

We are now ready to prove the main theorem of this paper.

\begin{theorem}{{\bf{(Main Theorem)}}}

Let $E$ be an elliptic curve over $K$ with $v(j(E))<0$. Then there exists an embedding $\phi:E\rightarrow{\mathbb{P}^{2}}$ such that its tropicalization contains a cycle of length $-v(j(E))$. 
\end{theorem}
\begin{proof}
We consider the embedding from Lemma \ref{EmbeddingLemma}. By Lemma \ref{LemmaValuation}, we have that $v(a)=0$ and $v(b)>0$. We then see that the tropicalization is given by the tropical polynomial
\begin{equation}
\mathrm{trop}(fg^2+2a'fg-2a'b-f^2g)=\mathrm{min}\{f+2g,f+g,v(b),2f+g\}.
\end{equation}
This tropicalization contains a cycle of length $3v(b)$, see Example \ref{MainExample}. We now calculate the $j$-invariant of $E$ using the formulas on \cite[Page 45]{Silv1}. This gives 
\begin{equation}\label{JinvariantEquation}
j(E)=\dfrac{-256\cdot{}a(a+6b)^3}{4ab^3 + 27b^4}%-64a^8 - 1152a^6b - 6912a^4b^2 - 13824a^2b^3}{a^2b^3 + 27/4b^4}.
\end{equation}
%The numerator of $j(E)$ has valuation zero, 
Since $v(a)=0$ is strictly less than $v(6\cdot{}b)=v(b)$, we have that $v(a+6b)=0$. Therefore, the numerator of $j(E)$ has valuation zero. Similarly, we find that the valuation of the denominator is $3v(b)$. This is the length of the cycle in the tropicalization of $\phi$, so we see that this embedding is numerically faithful. %By applying \cite[Corollary ...]{BPR11} to the Newton polygon of $fg^2+2a'fg-2a'b-f^2g$, we then also see that this tropicalization is faithful, as desired. 
%We consider the embedding from Lemma \ref{EmbeddingLemma}. By Lemma \ref{LemmaValuation}, we have that $v(a)=0$ and $v(b)>0$. We then see that the tropicalization is given by the tropical polynomial
%\begin{equation}
%\mathrm{trop}(fg^2+2a'fg-2a'b-f^2g)=\mathrm{min}\{f+2g,f+g,v(b),2f+g\}.
%\end{equation}
%This tropicalization contains a cycle of length $3v(b)$, see Example \ref{ExampleMainTheorem1}. We now calculate the $j$-invariant of $E$ using the formulas on \cite[Page 45]{Silv1}. This gives 
%\begin{equation}
%j(E)=\dfrac{-64a^8 - 1152a^6b - 6912a^4b^2 - 13824a^2b^3}{a^2b^3 + 27/4b^4}.
%\end{equation}
%The numerator of $j(E)$ has valuation zero, since $v(a^{8})=0$ is strictly less than the other valuations. Similarly, we find that the valuation of the denominator is $3v(b)$. This is the length of the cycle in the tropicalization of $\phi$, so we see that this embedding is faithful, as desired. 
\end{proof}

\bibliographystyle{alpha}
\bibliography{bibfiles}{}

\begin{thebibliography}{ABBR15}

\bibitem[ABBR15]{ABBR1}
Omid Amini, Matthew Baker, Erwan Brugall{\'{e}}, and Joseph Rabinoff.
\newblock Lifting harmonic morphisms {I}: {M}etrized complexes and {B}erkovich
  skeleta.
\newblock {\em Springer, Research in the Mathematical Sciences}, 2(1), June
  2015.

\bibitem[Ber71]{Bergman1971}
George~M. Bergman.
\newblock The logarithmic limit-set of an algebraic variety.
\newblock {\em Transactions of the American Mathematical Society},
  157:459--459, 1971.

\bibitem[Ber12]{berkovich2012}
V.G. Berkovich.
\newblock {\em Spectral Theory and Analytic Geometry over Non-Archimedean
  Fields}.
\newblock Mathematical Surveys and Monographs. American Mathematical Society,
  2012.

\bibitem[BG84]{Bieri1984}
Robert Bieri and J.R.J. Groves.
\newblock The geometry of the set of characters induced by valuations.
\newblock {\em Journal f\"{u}r die reine und angewandte Mathematik},
  347:168--195, 1984.

\bibitem[BPR14]{BPRa1}
Matthew Baker, Sam Payne, and Joseph Rabinoff.
\newblock On the structure of nonarchimedean analytic curves.
\newblock In {\em Tropical and Non-Archimedean Geometry}, volume 605, pages pp.
  93--121. American Mathematical Society, 2014.

\bibitem[BPR16]{BPR11}
Matthew Baker, Sam Payne, and Joseph Rabinoff.
\newblock Nonarchimedean geometry, tropicalization, and metrics on curves.
\newblock {\em Algebr. Geom.}, 3(1):63--105, January 2016.

\bibitem[BR10]{BakerRumely}
Matthew Baker and Robert Rumely.
\newblock {\em Potential Theory and Dynamics on the Berkovich Projective Line
  (Mathematical Surveys and Monographs)}.
\newblock American Mathematical Society, 2010.

\bibitem[BR14]{SkeletonJacobian}
Matthew Baker and Joseph Rabinoff.
\newblock {The Skeleton of the Jacobian, the Jacobian of the Skeleton, and
  Lifting Meromorphic Functions From Tropical to Algebraic Curves}.
\newblock {\em International Mathematics Research Notices},
  2015(16):7436--7472, 10 2014.

\bibitem[CHW14]{Cueto2014}
Maria~Angelica Cueto, Mathias H\"{a}bich, and Annette Werner.
\newblock Faithful tropicalization of the {G}rassmannian of planes.
\newblock {\em Mathematische Annalen}, 360(1-2):391--437, April 2014.

\bibitem[CLS11]{Cox2011}
David Cox, John Little, and Henry Schenck.
\newblock {\em Toric Varieties}.
\newblock American Mathematical Society, July 2011.

\bibitem[CS13]{CSHoneycomb}
Melody Chan and Bernd Sturmfels.
\newblock Elliptic curves in honeycomb form.
\newblock In {\em Algebraic and combinatorial aspects of tropical geometry :
  CIEM Workshop on Tropical Geometry, December 12-16, 2011}. American
  Mathematical Society, 2013.

\bibitem[Eis95]{Eisenbud1995}
David Eisenbud.
\newblock {\em Commutative Algebra}.
\newblock Springer New York, 1995.

\bibitem[EKL06]{Einsiedler2006}
Manfred Einsiedler, Mikhail Kapranov, and Douglas Lind.
\newblock Non-archimedean amoebas and tropical varieties.
\newblock {\em Journal f\"{u}r die reine und angewandte Mathematik (Crelles
  Journal)}, 2006(601), jan 2006.

\bibitem[Ful93]{FultonToric}
William Fulton.
\newblock {\em Introduction to Toric Varieties. (AM-131)}.
\newblock Princeton University Press, 1993.

\bibitem[Har77]{Hart1}
Robin Hartshorne.
\newblock {\em Algebraic Geometry}.
\newblock Springer New York, 1977.

\bibitem[Jel18]{JellMumfordFaithful}
Philipp Jell.
\newblock Constructing smooth and fully faithful tropicalizations for mumford
  curves.
\newblock {\em arXiv:1805.11594}, 2018.

\bibitem[Jon16]{Jonsson1}
Mattias Jonsson.
\newblock Topics in algebraic geometry: Berkovich spces.
\newblock {\em http://www-personal.umich.edu/~takumim/Berkovich.pdf}, 2016.

\bibitem[KMM09]{katz_markwig_markwig_2009}
Eric Katz, Hannah Markwig, and Thomas Markwig.
\newblock The tropical j-invariant.
\newblock {\em LMS Journal of Computation and Mathematics}, 12:275--294, 2009.

\bibitem[Kut17]{HypertoricFaithful}
Max Kutler.
\newblock Faithful tropicalization of hypertoric varieties.
\newblock {\em Thesis, University of Oregon}, 2017.

\bibitem[Liu06]{liu2}
Q.~Liu.
\newblock {\em Algebraic Geometry and Arithmetic Curves}.
\newblock Oxford Graduate Texts in Mathematics (Book 6). Oxford University
  Press, 2006.

\bibitem[MS15]{tropicalbook}
D.~Maclagan and B.~Sturmfels.
\newblock {\em Introduction to Tropical Geometry}.
\newblock Graduate Studies in Mathematics. American Mathematical Society, 2015.

\bibitem[Rab12]{Rabinoff2012}
Joseph Rabinoff.
\newblock Tropical analytic geometry, newton polygons, and tropical
  intersections.
\newblock {\em Advances in Mathematics}, 229(6):3192--3255, April 2012.

\bibitem[Sil94]{silv2}
Joseph~H. Silverman.
\newblock {\em Advanced Topics in the Arithmetic of Elliptic Curves}.
\newblock Springer New York, 1994.

\bibitem[Sil09]{Silv1}
Joseph~H. Silverman.
\newblock {\em The Arithmetic of Elliptic Curves}.
\newblock Springer New York, 2009.

\bibitem[Wag15]{Wagner2015}
Till Wagner.
\newblock Faithful tropicalization of {M}umford curves of genus two.
\newblock {\em Beitr\"{a}ge zur Algebra und Geometrie / Contributions to
  Algebra and Geometry}, 58(1):47--67, November 2015.

\bibitem[Zie95]{Ziegler1995}
G\"{u}nter~M. Ziegler.
\newblock {\em Lectures on Polytopes}.
\newblock Springer New York, 1995.

\end{thebibliography}

\end{document}